\documentclass[a4paper]{article}

\usepackage[T1]{fontenc}
\usepackage[latin9]{inputenc}
\usepackage{geometry}
\usepackage{amsmath}
\usepackage{amsthm}
\usepackage{amssymb}
\usepackage{bbm}
\usepackage{color}
\usepackage{paralist}
\usepackage[unicode=true,pdfusetitle,
 bookmarks=true,bookmarksnumbered=false,bookmarksopen=false,
 breaklinks=false,pdfborder={0 0 0},pdfborderstyle={},backref=false,colorlinks=false]
 {hyperref}

\makeatletter

\usepackage[nameinlink,capitalise,noabbrev]{cleveref}

\hypersetup{%
    bookmarksnumbered, bookmarksopen=true, bookmarksopenlevel=1,%
}

\theoremstyle{plain}
\newtheorem{thm}{Theorem}[section]
\crefname{thm}{Theorem}{Theorems}
\theoremstyle{plain}
\newtheorem{lem}[thm]{Lemma}
\crefname{lem}{Lemma}{Lemmas}
\theoremstyle{plain}

\theoremstyle{plain}
\newtheorem*{claim*}{Claim}
\crefname{claim}{Claim}{Claims}
\theoremstyle{definition}

\theoremstyle{plain}

\crefname{conjecture}{Conjecture}{Conjectures}
\theoremstyle{plain}
\newtheorem{prop}[thm]{Proposition}
\crefname{prop}{Proposition}{Propositions}
\theoremstyle{definition}

\theoremstyle{definition}

\theoremstyle{plain}
\newtheorem{claim}[thm]{Claim}
\newtheorem{fact}[thm]{Fact}
\crefname{fact}{Fact}{Facts}
\crefformat{equation}{#2(#1)#3}

\crefname{subsection}{Subsection}{Subsections}

\date{}

\crefformat{equation}{#2(#1)#3}

\let\originalleft\left
\let\originalright\right
\renewcommand{\left}{\mathopen{}\mathclose\bgroup\originalleft}
\renewcommand{\right}{\aftergroup\egroup\originalright}
\usepackage{verbatim}

\makeatletter
\renewcommand*{\UrlTildeSpecial}{%
  \do\~{%
    \mbox{%
      \fontfamily{ptm}\selectfont
      \textasciitilde
    }%
  }%
}%
\let\Url@force@Tilde\UrlTildeSpecial
\makeatother

\let\OLDthebibliography\thebibliography
\renewcommand\thebibliography[1]{
  \OLDthebibliography{#1}
  \setlength{\parskip}{0pt}
  \setlength{\itemsep}{3pt plus 0.3ex}
}

\makeatother

\allowdisplaybreaks
\numberwithin{equation}{section}

\begin{document}
\global\long\def\per{\operatorname{per}}%
\global\long\def\RR{\mathbb{R}}%
\global\long\def\FF{\mathbb{F}}%
\global\long\def\QQ{\mathbb{Q}}%
\global\long\def\E{\mathbb{E}}%
\global\long\def\Var{\operatorname{Var}}%
\global\long\def\CC{\mathbb{C}}%
\global\long\def\NN{\mathbb{N}}%
\global\long\def\ZZ{\mathbb{Z}}%
\global\long\def\GG{\mathbb{G}}%
\global\long\def\tallphantom{\vphantom{\sum}}%
\global\long\def\tallerphantom{\vphantom{\int}}%
\global\long\def\supp{\operatorname{supp}}%
\global\long\def\one{\mathbbm{1}}%
\global\long\def\d{\operatorname{d}}%
\global\long\def\Unif{\operatorname{Unif}}%
\global\long\def\Po{\operatorname{Po}}%
\global\long\def\Bin{\operatorname{Bin}}%
\global\long\def\Ber{\operatorname{Ber}}%
\global\long\def\Geom{\operatorname{Geom}}%
\global\long\def\Rad{\operatorname{Rad}}%
\global\long\def\floor#1{\left\lfloor #1\right\rfloor }%
\global\long\def\ceil#1{\left\lceil #1\right\rceil }%
\global\long\def\falling#1#2{\left(#1\right)_{#2}}%
\global\long\def\cond{\,\middle|\,}%
\global\long\def\su{\subseteq}%
\global\long\def\row{\operatorname{row}}%
\global\long\def\col{\operatorname{col}}%
\global\long\def\spn{\operatorname{span}}%
\global\long\def\eps{\varepsilon}%

\global\long\def\ls#1{\textcolor{blue}{\textbf{[LS comments:} #1\textbf{]}}}
\global\long\def\mk#1{\textcolor{red}{\textbf{[MK comments:} #1\textbf{]}}}

\title{On the permanent of a random symmetric matrix}
\author{
Matthew Kwan\thanks{Department of Mathematics, Stanford University, Stanford, CA.
Email: \href{mattkwan@stanford.edu}{\nolinkurl{mattkwan@stanford.edu}}.
Research supported by NSF Award DMS-1953990.}
\and
Lisa Sauermann\thanks{School of Mathematics, Institute for Advanced Study, Princeton, NJ. Email: \href{lsauerma@mit.edu}{\nolinkurl{lsauerma@mit.edu}}. Research supported by NSF Grant CCF-1900460 and NSF Award DMS-1953772. Part of this work was completed while this author was a Szeg\H{o} Assistant Professor at Stanford University.}
}

\maketitle

\begin{abstract}\noindent
Let $M_{n}$ denote a random symmetric $n\times n$ matrix, whose entries on and above the diagonal are i.i.d.\ Rademacher random variables (taking
values $\pm1$ with probability $1/2$ each). Resolving a conjecture
of Vu, we prove that the permanent of $M_{n}$ has magnitude $n^{n/2+o(n)}$
with probability $1-o(1)$. Our result can also be extended
to more general models of random matrices. In our proof, we build on and extend some techniques introduced by Tao and Vu, studying the evolution of permanents of submatrices in a random matrix process.

\vspace{0.2cm}

\end{abstract}

\section{Introduction}

Two of the most basic matrix parameters are the \emph{determinant}
and the \emph{permanent}: for an $n\times n$ matrix $M=(x_{i,j})_{i,j}$,
define
\begin{equation}
\det(M)=\sum_{\pi\in S_{n}}\operatorname{sign}(\pi)\prod_{i=1}^{n}x_{i,\pi(i)}\quad\textnormal{and}\quad\per(M)=\sum_{\pi\in S_{n}}\prod_{i=1}^{n}x_{i,\pi(i)}.\label{eq:det-per}
\end{equation}
A central direction of research in probabilistic combinatorics and random matrix
theory is to understand the determinant and permanent of different
types of random matrices. For example, let $A_{n}$ be an $n\times n$ matrix whose entries are i.i.d.\ Rademacher-distributed
random variables, taking values $\pm1$ with probability $1/2$ each (this is often called a \emph{random
Bernoulli matrix}).
A classical theorem of Koml\'os~\cite{Kom67} (perhaps the foundational
theorem in discrete random matrix theory) is that $\Pr(\det A_{n}=0)=o(1)$
as $n\to\infty$. That is to say, $A_n$ is \emph{asymptotically
almost surely} nonsingular. Since then, there has been intensive effort
to refine our understanding of the singularity probability
(see \cite{KKS95,TV06,TV07,RV08,BVW10}), culminating in a recent breakthrough of Tikhomirov~\cite{Tik20} proving that $\Pr(\det A_{n}=0)=2^{-n+o(n)}$.
The problem of estimating the order of magnitude of $\det A_{n}$ has also received significant attention:
Tao and Vu~\cite{TV06} proved that with probability $1-o(1)$
we have $\left|\det A_{n}\right|=n^{n/2+o(n)}$, and later
Nguyen and Vu~\cite{NV14} proved a central limit theorem for $\log |\det A_{n}|$ (see also \cite{Gir79,Gir97}).

Most of the above-mentioned results generalise readily to more general types
of random matrices, where the entries are independently sampled from
any subgaussian distribution. There has also been intensive
interest in random matrices with dependence between the entries. Perhaps the most prominent examples are \emph{symmetric} random matrices. Let $M_{n}$ be the random matrix whose entries on and above the diagonal are independent Rademacher random variables, and the entries below the diagonal are chosen to make the matrix symmetric (equivalently, we can choose a random matrix $A_n$ as above and condition on the event that $A_{n}$ is symmetric). The study of random symmetric matrices has necessitated the development of new tools, but by now there is a fairly complete understanding of the determinant of a random symmetric matrix with Rademacher entries. The fact that $\Pr(\det M_{n}=0)=o(1)$ was first proved by Costello, Tao and Vu~\cite{CTV06} (see also \cite{Fer20}), and stronger
estimates on $\Pr(\det M_{n}=0)$ were obtained by several
authors~\cite{Ngu12b,Ver14,FJ19,CMMM19}. It is also known that  with probability $1-o(1)$ we have $|\det M_{n}|=n^{n/2+o(n)}$
(this follows from work on the least singular value of $M_{n}$ due
to Nguyen~\cite{Ngu12a} and Rudelson~\cite{Ver14}, together with Wigner's celebrated
semicircle law~\cite{Wig55,Wig58}), and a central limit theorem for $\log |\det M_{n}|$
was proved by Bourgade and Mody~\cite{BM19} (see also \cite{TV12}).

It is widely believed that for all the above-mentioned theorems concerning determinants of random matrices (symmetric or not), there should be analogous theorems for permanents. However, the permanent appears to be a much more challenging parameter to study. For example, while the determinant encodes information about linear dependence and can be interpreted as the (signed) volume of a certain parallelepiped, it only seems possible to attack the permanent from a ``combinatorial'' point of view, directly considering the formal definition in \cref{eq:det-per}. The fact that the permanent is harder to study is maybe not surprising, since (in contrast to the determinant) the permanent of a matrix is $\#$P-hard to compute (as was famously proved by Valiant~\cite{Val79}). Even the analogue of the singularity problem, to show that $\Pr(\per A_n=0)=o(1)$, was open for a long time. In 2009, Tao and Vu~\cite{TV09} finally resolved this problem and also estimated the typical magnitude of $\per A_{n}$: namely, they proved that asymptotically almost surely $\left|\per A_{n}\right|=n^{n/2+o(n)}$. Perhaps surprisingly, permanents of random matrices turn out to be of interest in quantum computing: Aaronson and Arkhipov~\cite{AA13} proved that quantum computers cannot be efficiently simulated by classical computers, conditional on a conjecture which strengthens the aforementioned Tao--Vu permanent theorem (see also \cite{Aar10,EM18,LM19}).

In this paper we study the permanent of random \emph{symmetric} matrices. More precisely, we estimate the typical magnitude of the permanent of a random symmetric matrix with Rademacher entries.
\begin{thm}\label{thm:per-magnitude}
\label{thm:main-0}Asymptotically almost surely, $|\per M_n|=n^{n/2+o(n)}$.
\end{thm}

The study of the permanent of a random symmetric matrix seems to have first been explicitly suggested by Tao and Vu (see \cite[Remark~1.6]{TV09}). They observed that their arguments for the permanent of a (not necessarily symmetric) random matrix ``do not seem to easily yield any non-trivial result for the permanent of a random \emph{symmetric} Bernoulli matrix''. The statement of \cref{thm:per-magnitude} has been conjectured by Vu in 2009 (see \cite{Vu09}). He also mentioned the conjecture in a recent survey \cite[Conjecture 6.11]{Vu20}, and described it as ``the still missing piece of the picture'' regarding determinants and permanents of random discrete matrices.

\cref{thm:per-magnitude} is actually a combination of two different results. First, the following proposition gives an upper bound on $|\per M_n|$.

\begin{prop}\label{thm:var}
For any $\varepsilon>0$, if $n$ is sufficiently large with respect to $\varepsilon$, we have
\[\Pr\left(|\per M_{n}|\ge n^{n/2+\varepsilon n}\right)\le n^{-\varepsilon n}.\]
\end{prop}

\cref{thm:var} is easily proved using an estimate for $\E[ (\per M_n)^2]$ and Markov's inequality; see \cref{sec:var}. The main role of this paper is to prove the following lower bound on $|\per M_n|$.

\begin{thm}
\label{thm:main}There is a positive constant $c>0$ such that for any  $\varepsilon>0$ the following holds. If $n$ is sufficiently large  with respect to $\eps$, we have
\[\Pr\left(|\per M_{n}|\le n^{n/2-\varepsilon n}\right)\le n^{-c}.\]
\end{thm}

We remark that the constant $c$ in \cref{thm:main} can be made explicit. For example, $c=1/150$ certainly suffices, though this can be improved substantially simply by optimising constants throughout the proof. However, without new ideas our methods do not seem to be capable of proving the statement of \cref{thm:main} with any $c\ge 1/2$. This state of affairs is essentially the same as for the non-symmetric case previously considered by Tao and Vu~\cite{TV09}.

Of course, \cref{thm:main} also shows that the probability of having $\per M_n=0$ is polynomially small. Before this paper no nontrivial bounds for this probability were known, except when $n=2^m-1$ with $m\in \NN$, where for elementary number-theoretic reasons it is actually impossible to have $\per M_n=0$ (see \cite{SS83}\footnote{Actually, this result is part of an extensive body of research concerning permanents of (non-random) matrices with $\pm1$ entries; see for example \cite{Wan74,BG19,BGT15,BGT17,SS83,Wan05,KS83,Sei84,KS84}.}).

Finally, we remark that the methods used to prove \cref{thm:main-0} are quite robust, and analogous theorems can be proved for much more general distributions. For example, consider any fixed real probability distributions $\mu$ and $\nu$, and let $M_n^{\mu,\nu}$ be the random symmetric matrix whose diagonal entries have distribution $\nu$ and whose off-diagonal entries have distribution $\mu$ (and whose entries on and above the diagonal are mutually independent). With some fairly mild assumptions on $\mu$ and $\nu$ it is a routine matter to prove an analogue of \cref{thm:var} for $M_n^{\mu,\nu}$, and in \cref{sec:concluding} we sketch how to make some minor adaptations to the proof of \cref{thm:main} to obtain a version that holds for very general distributions (we only require that $\mu$ has nontrivial support). In particular, one can prove an analogue of \cref{thm:main-0} for random symmetric Gaussian matrices such as the Gaussian Orthogonal Ensemble (GOE).

\textbf{Notation.} In this paper, we use the notation $\NN=\{1,2,\dots\}$ for the positive integers. All logarithms are to base $e$. For functions $f:\NN\to \RR$ and $g:\NN\to \RR_{>0}$, we write $f=o(g)$ if $f(n)/g(n)\to 0$ as $n\to\infty$.

\section{Proof of \texorpdfstring{\cref{thm:var}}{Proposition~\ref{thm:var}}: the second moment of the permanent}\label{sec:var}

In this section we provide the simple proof of \cref{thm:var}. It will be an immediate consequence of Markov's inequality and the following lemma.

\begin{lem}\label{lem:var}
$\E[ (\per M_n)^2]\le n^{n+o(n)}$.
\end{lem}

\begin{proof}
Write $x_{i,j}$ for the $(i,j)$ entry of $M_{n}$. For a permutation
$\pi\in S_n$, let $X_{\pi}=\prod_{i=1}^n x_{i,\pi(i)}$,
so that $\per M_{n}=\sum_{\pi}X_{\pi}$. It will not be necessary for the proof, but we remark that
\[
\E X_{\pi}=\begin{cases}
1 & \text{if }\pi\text{ consists only of 2-cycles,}\\
0 & \text{otherwise}.
\end{cases}
\]
Furthermore, let $I_{\pi}\subseteq\{1,\dots,n\}$
be the set of indices which appear in 2-cycles of $\pi$, and let
$F_{\pi}$ be the family of sets $\{i,\pi(i)\}$, for
$i\notin I_{\pi}$. Then for two permutations $\pi,\pi'\in S_n$,
we have 
\[
\E [X_{\pi}X_{\pi'}]=\begin{cases}
1 & \text{if }(I_{\pi},F_{\pi})=(I_{\pi'},F_{\pi'}),\\
0 & \text{otherwise}.
\end{cases}
\]
For $k=0,\dots,n$, let $\mathcal{Q}_{k}$ be the set of all permutations $\pi\in S_n$
satisfying $|I_{\pi}|=k$, and note that $|\mathcal{Q}_{k}|\le \binom{n}{k}k^{k/2}(n-k)!\le 2^nk^{k/2}n^{n-k}$. Indeed, for any choice of a set $I\subseteq \{1,\dots,n\}$ of size $k$, there are at most $k^{k/2}$ ways to partition $I$ into 2-cycles, and at most $(n-k)!$ ways to choose a permutation of $\{1,\dots,n\}\setminus I$.

Now, for any $0\leq k\leq n$ and any $\pi\in\mathcal{Q}_{k}$, there are at most $2^{n-k}k^{k/2}$
choices of $\pi'$ satisfying $(I_{\pi},F_{\pi})=(I_{\pi'},F_{\pi'})$.
Indeed, for such $\pi'$, the restriction of $\pi'$ to $I_{\pi}$
must be a permutation of $I_\pi$ consisting only of 2-cycles (so there are at most $k^{k/2}$ ways to choose this restriction), and for each $i\notin I_{\pi}$ we
must have $\pi'(i)\in\{\pi(i),\pi^{-1}(i)\}$. It follows that 
\[
\E[(\per M_{n})^2]=\sum_{\pi,\pi'\in S_n}\E X_{\pi}X_{\pi'} \le\sum_{k=0}^n|\mathcal{Q}_{k}|\cdot2^{n-k}k^{k/2}\le\sum_{k=0}^n 4^n k^k n^{n-k} \le n^{n+o(n)},
\]
as claimed.
\end{proof}

\begin{proof}[Proof of \cref{thm:var}]
Let $\eps>0$, and suppose that $n$ is sufficiently large such that the $o(n)$-term in \cref{lem:var} is at most $\eps n$. Then we have $\E[ (\per M_n)^2]\le n^{n+\eps n}$ and consequently by Markov's inequality
\[\Pr\left(\vert\per M_n\vert\geq n^{n/2+\eps n}\right)=\Pr\left((\per M_n)^2\geq n^{n+2\eps n}\right)\leq \frac{\E[ (\per M_n)^2]}{n^{n+2\eps n}}\leq n^{-\eps n},\]
as desired.
\end{proof}

\section{Structure of the proof of \texorpdfstring{\cref{thm:main}}{Theorem~\ref{thm:main}}}

The rest of this paper is devoted to the proof of \cref{thm:main}. In this section we outline the high-level strategy of the proof, stating two key lemmas and deducing \cref{thm:main} from them.

We couple the distributions of the matrices $M_n$ for all $n\in \NN$ by viewing each $M_{n}$ as containing the first $n$ rows and columns
of an infinite random symmetric matrix (with Rademacher entries). Say that subsets of a given ground
set are \emph{complement-disjoint} if their complements are disjoint.
For $A,B\su \{1,\dots,n\}$, let $M_{n}[A,B]$ be the submatrix of $M_{n}$ consisting
of the rows in $A$ and the columns in $B$. For $\lambda\geq 0$, we say that a matrix is
\emph{$\lambda$-heavy} if its permanent has absolute value at least
$\lambda$.

The following lemma shows that with high probability there exists a heavy submatrix of $M_n$ consisting of almost all the rows and columns of $M_n$, and moreover we have some control over which rows and columns are not included. Roughly speaking, it is proved by studying how permanents of submatrices evolve in the sequence of random matrices $M_1,\dots,M_n$.

\begin{lem}
\label{lem:grow-single-minor}There is a positive constant $c>0$ such that for any  $\varepsilon>0$ the following holds. Let $n\in \NN$ be sufficiently large with respect to $\eps$, and let $L=\lfloor (\log n)/10\rfloor$. Let $X$ and $Y$ be disjoint subsets of $\{ 1,\dots,n\}$ with sizes $\vert X\vert=L$ and $\vert Y\vert=3L$. Then with probability at least $1-(1/4)\cdot n^{-c}$ there is a set $B$ satisfying $|B|=n-L$ and $\{ 1,\dots,n\} \setminus Y\su B\su \{ 1,\dots,n\}$, such that $M_{n}[\{ 1,\dots,n\} \setminus X,B]$ is $n^{(1-\eps)n/2}$-heavy.
\end{lem}

By applying \cref{lem:grow-single-minor} with various different choices of $X$ and $Y$, we can obtain many heavy submatrices $M_n[A_1,B_1],\dots, M_n[A_m,B_m]$ such that  the sets $A_{1},\dots,A_{m},B_{1},\dots,B_{m}$ are complement-disjoint. The next lemma states that in such a situation, if we sample $M_{n+1}$ by adding a random row and column to $M_n$, then a large proportion of our submatrices can be transformed into larger submatrices without losing much of their heaviness. We will apply this lemma repeatedly, each step decreasing by 1 the number of rows and columns that our submatrices are missing.

\begin{lem}
\label{lem:endgame-step}
Let $m\in \NN$ be sufficiently large. Let $\lambda>0$, let $1\leq L< n$ be integers, and let $A_{1},\dots,A_{m},B_{1},\dots,B_{m}$
be complement-disjoint subsets of $\{ 1,\dots,n\} $ of
size $n-L$. Let us condition on an outcome of $M_{n}$ such that all the submatrices $M_{n}[A_{\ell},B_{\ell}]$, for $\ell=1,\dots,m$, are $\lambda$-heavy.

Then, with
probability at least $1-m^{-1/24}$, for $m'=\ceil{m/36}$ there are complement-disjoint subsets $A_{1}',\dots,A_{m'}',B_{1}',\dots,B_{m'}'\subseteq\{ 1,\dots,n+1\} $
of size $n-L+2$, such that for all $\ell=1,\dots,m'$ the submatrices $M_{n+1}[A_{\ell}',B_{\ell}']$ are $\lambda/(4n^{4})$-heavy.
\end{lem}

We now show how to deduce \cref{thm:main} from \cref{lem:grow-single-minor,lem:endgame-step}.

\begin{proof}[Proof of \cref{thm:main}]
Choose an absolute constant $0<c<1/50$, such that the statement in \cref{lem:grow-single-minor} is satisfied. Fix $\eps>0$.

Let $L=\lfloor (\log n)/10\rfloor$ and $m=\lfloor (n-L)/(4L)\rfloor$, and consider disjoint sets $X_{1},\dots,X_{m},Y_{1},\dots,Y_{m}\subseteq\{ 1,\dots,n-L\} $ with $\vert X_1\vert=\dots=\vert X_m\vert=L$ and $\vert Y_1\vert=\dots=\vert Y_m\vert=3L$.

For each $\ell=1,\dots, m$, we apply \cref{lem:grow-single-minor} to  the subsets $X_{\ell},Y_{\ell}\subseteq \{ 1,\dots,n-L\}$. Each application fails with probability at most $n^{-c}/4$, so it follows from Markov's inequality (see for example \cref{lem:Markov}) that with probability at least $1-(1/2)\cdot n^{-c}$, at least $m/2$ applications succeed. That is to say, with $m'= \ceil{m/2}$ and $\lambda=(n-L)^{(1-\eps)(n-L)/2}$, we obtain complement-disjoint sets $A_{1},\dots,A_{m'},B_{1},\dots,B_{m'}\subseteq\{ 1,\dots,n-L\} $ of size $n-2L$ such that for all $\ell=1,\dots, m'$, the matrices $M_{n-L}[A_{\ell},B_{\ell}]$  are $\lambda$-heavy. Note that if $n$ is sufficiently large with respect to $\eps$, then $m'\ge n^{9/10}$ and $\lambda\geq n^{n/2-(3/4)\eps n}$.

Now, we wish to iteratively apply \cref{lem:endgame-step}, $L$ times in total. After each application, $m'$ decreases by a factor of $36< e^4$, so after $L=\lfloor (\log n)/10\rfloor$ steps the value of $m'$ will still be at least $\sqrt{n}$. Each of the $L$ applications of \cref{lem:endgame-step} succeeds with probability at least $1-n^{-1/48}$. Thus, with probability at least $1-L\cdot n^{-1/48}\geq 1-(1/2)\cdot n^{-c}$ (for sufficiently large $n$) we can indeed apply the lemma $L$ times. In the end we obtain subsets $A',B'\subseteq \{ 1,\dots,n\}$ of size $n$ such that the matrix $M_n[A',B']$ is $\lambda'$-heavy, where
\[\lambda'=\frac{\lambda}{4(n-L)^{4}\cdot 4(n-L+1)^{4}\dotsm 4(n-1)^{4}}\geq \frac{\lambda}{(4n^{4})^L}\geq \frac{\lambda}{n^{5L}}\geq n^{n/2-(3/4)\eps n-\log n/2}\geq n^{n/2-\eps n}\]
(again assuming that $n$ is sufficiently large with respect to $\eps$). But note that we must have $A'=B'=\{ 1,\dots,n\}$, so this means that $M_n$ itself is $\lambda'$-heavy. In summary, with probability at least $1-n^{-c}$ we have $\vert \per M_n\vert\geq \lambda'\geq n^{n/2-\eps n}$, as desired.
\end{proof}

We remark that the overall structure of our proof is similar to the work of Tao and Vu~\cite{TV09} on permanents of (not necessarily symmetric) random matrices. Indeed, Tao and Vu's proof can also be broken up into two parts analogous to \cref{lem:grow-single-minor,lem:endgame-step}. However, in Tao and Vu's setting, all entries of the random matrix are independent, allowing them to expose the entries row by row. After exposing $k$ rows, they consider $k\times k$ submatrices that consist of all the $k$ exposed rows (and of $k$ of the $n$ columns). When exposing the $k$-th row, the permanent of any such $k\times k$ submatrix can be described as a linear polynomial in some of the entries of the new row, where the coefficients are given by the permanents of certain $(k-1)\times (k-1)$ submatrices in the first $k-1$ rows. In contrast, in our setting with the random symmetric matrix $M_n$,  we are forced to expose the entries of our matrix in a different way: at the $k$-th step we reveal the entries in $M_k$ that are not present in $M_{k-1}$ (that is, we add a new random row and column, with equal entries, to the matrix considered so far\footnote{This type of exposure is standard in the study of symmetric random matrices (see for example \cite{CTV06}).}). Since there is only one $k\times k$ submatrix in $M_k$ (namely $M_k$ itself), in our setting we also need to consider the permanents of (substantially) smaller submatrices of $M_k$.

This more intricate strategy introduces significant challenges. Most notably, the permanents of the submatrices of $M_k$ are described by \emph{quadratic} polynomials in the new matrix entries, where the coefficients depend on the permanents of certain submatrices of $M_{k-1}$ (this is in contrast to Tao and Vu's setting, where the permanents are described by linear polynomials in the entries of the new row). This necessitates the use of some more sophisticated probabilistic tools. Furthermore, there can be certain types of cancellations within these quadratic polynomials, which are not possible for the linear polynomials in the Tao--Vu setting. For example, even if all submatrices of $M_{k-1}$ have non-zero permanent, it can happen that the polynomial describing the permanent of some submatrix of $M_k$ has only very few nonzero coefficients. Handling these types of cancellations requires key new ideas.

\textbf{Organization of the rest of the paper.} \cref{lem:grow-single-minor} will be proved in \cref{sec:grow-single-minor}, and \cref{lem:endgame-step} will be proved in \cref{sec:endgame-step}. As preparation, in Section 4 we collect some probabilistic tools that we will use in the proofs, and in Section 5 we collect some lemmas that can be obtained by studying permanent expansion formulas.

\section{Probabilistic tools}\label{sec:tools}

This section collects some theorems and simple facts that will be needed for proving \cref{lem:grow-single-minor,lem:endgame-step}. We start with some basic anti-concentration estimates for linear forms. The first of these is the famous Erd\H os--Littlewood--Offord inequality (see for example \cite[Corollary~7.8]{TV10}).
\begin{thm}
\label{lem:LO}Let $t\ge 1$ be a real number, and let $f$ be a linear
polynomial in $n$ variables, in which at least $m$ degree-$1$ coefficients have absolute value at least $r$. Then for uniformly random $\boldsymbol{\xi}\in\{ -1,1\} ^{n}$ we have
\[
\Pr\left(|f(\boldsymbol{\xi})|\le t\cdot r\right)\leq (\lceil t\rceil +1)\cdot \binom{m}{\lfloor m/2\rfloor}\cdot 2^{-m}\leq \frac {3t}{\sqrt{m}}.
\]
\end{thm}

We will also need the following very easy fact.
\begin{fact}
\label{fact:non-degenerate-linear}Let $f$ be a linear polynomial in
$n$ variables, which has at least one coefficient
with absolute value at least $r$. Then for uniformly random $\boldsymbol{\xi}\in\{ -1,1\} ^{n}$
we have
\[
\Pr\left(|f(\boldsymbol{\xi})|<r\right)\le\frac12.
\]
\end{fact}
\begin{proof}
First, suppose that the constant coefficient of $f$ has absolute value at least $r$, and suppose without loss of generality that $f(\boldsymbol{\xi})=a_1\xi_1+\dots+a_n\xi_n+c$ with $c\geq r$. Then, by symmetry we have that $\Pr(a_1\xi_1+\dots+a_n\xi_n< 0)\leq 1/2$ and therefore $\Pr(f(\boldsymbol{\xi})<r)\leq 1/2$.

Otherwise, for some $i\in \{1,\dots,n\}$, the coefficient of $\xi_i$ in $f(\boldsymbol{\xi})$ has absolute value at least $r$, and suppose without loss of generality that $i=1$. Condition on any outcomes of the variables $\xi_2,\dots,\xi_{n}$, and observe that then we can have $|f(\boldsymbol{\xi})|<r$ for at most one of the two possible outcomes of $\xi_1$.
\end{proof}

We will also need counterparts of both the above statements for quadratic polynomials. The quadratic counterpart of \cref{fact:non-degenerate-linear} is again easy to prove.
\begin{fact}

\label{fact:non-degenerate}Let $f$ be a quadratic polynomial in
$n$ variables, which has at least one multilinear degree-2 coefficient
with absolute value at least $r$. Then for uniformly random $\boldsymbol{\xi}\in\{ -1,1\} ^{n}$
we have
\[
\Pr(|f(\boldsymbol{\xi})|<r)\le\frac34.
\]
\end{fact}

\begin{proof}
We may assume that $f$ is multilinear (every term of the form $\xi_i^2$ can be replaced by the constant $1$ without changing the behaviour of $f(\boldsymbol \xi)$).
Suppose without loss of generality that the coefficient $a_{12}$ of $\xi_{1}\xi_{2}$ satisfies $\vert a_{12}\vert \geq r$, and write $f(\xi_{1},\dots,\xi_{n})=\xi_{1}\cdot (a_{12} \xi_{2}+g(\xi_{3},\dots,\xi_{n}))+h(\xi_{2},\dots,\xi_{n})$.
Conditioning on any outcomes of $\xi_{3},\dots,\xi_{n}$, with probability
at least $1/2$ we have $|a_{12}\xi_{2}+g(\xi_{3},\dots,\xi_{n})|\ge r$.
Then, conditioning on such an outcome of $\xi_{2}$, we have $\vert f(\boldsymbol{\xi})\vert \ge r$
with probability at least $1/2$.
\end{proof}

It is more delicate to generalise the Erd\H os--Littlewood--Offord inequality to quadratic polynomials. For a multilinear quadratic polynomial $f$ in the variables $x_1,\dots,x_n$ and for a real number $r> 0$, let $G^{(r)}(f)$
be the graph with vertex set $\{1,\dots,n\}$ having an edge $ij$ whenever the coefficient of $x_ix_j$ in $f$ has absolute value at least $r$. Let $\nu(G)$ be the matching number\footnote{The matching number of a graph $G$ is the largest number $\nu$ such that one can find $\nu$ disjoint edges in $G$.} of a graph $G$. The following is a special case of a theorem proved by Meka, Nguyen and Vu~\cite[Theorem 1.6]{MNV16}.

\begin{thm}
\label{lem:MNV}Let $r>0$, let $f$ be a multilinear quadratic polynomial in $n$ variables, and let $\nu=\nu(G^{(r)}(f))\geq 3$. Then for uniformly random $\boldsymbol{\xi}\in\{ -1,1\} ^{n}$
we have
\[
\Pr(|f(\boldsymbol{\xi})|\le r)\le \frac{(\log \nu)^C}{\nu^{1/2}},
\]
where $C$ is an absolute constant.
\end{thm}

The following concentration inequality is a special case of the Azuma--Hoeffding martingale concentration inequality, and is sometimes known as McDiarmid's inequality (see for example \cite[Lemma~1.34]{TV10}).

\begin{lem}
\label{lem:AH}Let $c>0$ and let $X$ be a random variable defined
in terms of independent random variables $\xi_{1},\dots,\xi_{n}$,
having the property that varying any individual $\xi_{i}$ affects
the value of $X$ by at most $c$. Then for any $t\geq 0$ we have
\[
\Pr\left(|X-\E X|\ge t\right)\le 2e^{-t^{2}/(2nc^{2})}.
\]
\end{lem}

The next inequality is a one-sided version of the Azuma--Hoeffding inequality for supermartingales (see \cite[Lemma~2.3]{TV09}).

\begin{lem}
\label{lem:AH2}Let $c>0$. In a probability space, let $Z_1,\dots,Z_n$ be a sequence of random objects, and let $W_1,\dots,W_n$ be a sequence of random variables, such that for each $k$, all of $Z_1,\dots,Z_{k},W_1,\dots,W_k$ are fully determined by $Z_k$, and such that $|W_{k+1}-W_{k}|\le c$ for all $k=1,\dots,n-1$. Suppose that the supermartingale property $\E[W_{k+1}|Z_k]\le W_k$ is satisfied for $k=1,\dots,n-1$. Then for any $t>0$ we have
\[
\Pr\left(W_n-W_1\ge t\right)\le e^{-t^{2}/(2nc^{2})}.
\]
\end{lem}

Recall that for $0<p<1$, a Bernoulli random variable $\chi\sim \Ber(p)$ is a random variable taking values $0$ and $1$ with $\Pr(\chi=1)=p$ and $\Pr(\chi=0)=1-p$. The following lemma is a version of the Chernoff concentration bound for sums of Bernoulli random variables (see for example \cite[Theorem~A.1.4]{AS}).

\begin{lem}
\label{lem:Chernoff}Let $\chi_1,\dots,\chi_m$ be independent Bernoulli random variables, where for each $i=1,\dots, m$ we have $\chi_i\sim \Ber(p_i)$ for some $0<p_i<1$. Then for any $t>0$ the sum $X=\chi_1+\dots+\chi_m$ satisfies
\[
\Pr\left(X-\E X> t\right)< e^{-2t^{2}/n}\quad \text{and}\quad \Pr\left(X-\E X< -t\right)< e^{-2t^{2}/n}.
\]
\end{lem}

Sometimes we will encounter random variables that \emph{stochastically dominate} a sum of Bernoulli random variables (we say that a random variable $X$ stochastically dominates another random variable $Y$ if there is a coupling of $X$ and $Y$ such that we always have $X\ge Y$). For example, consider a random process with $n$ steps where each step satisfies a certain property with probability at least $1/2$, even when conditioning on any outcome of the previous steps. Then the number $X$ of steps with this property stochastically dominates a sum of $n$ independent $\Ber(1/2)$ random variables. Denoting this sum by $Y$, we therefore have $\Pr\left(X-(n/2)< -t\right)\leq \Pr\left(Y-(n/2)< -t\right)< e^{-2t^{2}/n}$ for any $t>0$ by \cref{lem:Chernoff}. Hence we can use \cref{lem:Chernoff} to show that a random variable is very likely reasonably large if it stochastically dominates a sum of Bernoulli random variables.

Finally, the following lemma is an easy consequence of Markov's inequality (see, for example, \cite[Lemma~2.1]{TV09}).

\begin{lem}
\label{lem:Markov}Let $1>p>q>0$, and let $E_1,\dots,E_m$ be events (not necessarily independent), each of which occurs with probability at least $p$. Then the probability that at least $qm$ of the events $E_1,\dots,E_m$ occur simultaneously is at least $(p-q)/(1-q)$.
\end{lem}

\section{Permanent expansion formulas}

Just as for the determinant, it is possible to expand the permanent of a matrix in terms of permanents of submatrices.
Below we record two such expansions, which we will use in the proofs of \cref{lem:grow-single-minor,lem:endgame-step}.
\begin{fact}
\label{fact:per-expansion}Let $M$ be an $n\times n$ matrix. Add
a new row $(x_{1},\dots,x_{n})$ to obtain an $(n+1)\times n$
matrix $M'$. Then for any subsets $A,B\subseteq\{ 1,\dots,n\} $
with $|B|=|A|+1$, we have
\[
\per M'[A\cup\{ n+1\} ,B]=\sum_{i\in B}x_{i}\per M[A,B\setminus\{ i\} ].
\]
\end{fact}

For a matrix $M$, let $M^{(i,j)}$ be the submatrix of
$M$ obtained by removing row $i$ and column $j$.
\begin{fact}
\label{fact:per-double-expansion}Let $M$ be an $n\times n$ matrix.
Add a new row $(x_{1},\dots,x_{n},z)$ and a new column
$(y_{1},\dots,y_{n},z)$ to obtain an $(n+1)\times(n+1)$
matrix $M'$. Then for any subsets $A,B\subseteq\{ 1,\dots,n\} $ with $|A|=|B|$, we have
\[
\per M'[A\cup\{ n+1\} ,B\cup\{ n+1\} ]=z\per M[A,B]+\sum_{i\in A,j\in B}x_{j}y_{i}\per M[A,B]^{(i,j)}.
\]
\end{fact}

We will use \cref{fact:per-expansion} in combination with the linear anti-concentration inequalities in \cref{fact:non-degenerate-linear} and \cref{lem:LO}, and we will use \cref{fact:per-double-expansion} in combination with the quadratic anti-concentration inequalities in  \cref{fact:non-degenerate} and \cref{lem:MNV}

Observe in particular that the formula in \cref{fact:per-double-expansion} gives an expression for
the permanent of a $(k+1)\times(k+1)$ matrix
in terms of permanents of $(k-1)\times(k-1)$
submatrices (and one $k\times k$ submatrix). This means that, for example, when we add a new row and
column to a matrix, the size of the largest submatrix with nonzero permanent can
increase by two. This observation will be crucial for the proof of \cref{lem:endgame-step}.

In the proof of \cref{lem:endgame-step}, we will apply \cref{fact:per-double-expansion} with the symmetric matrices $M=M_{n}$
and $M'=M_{n+1}$. That is to say, $(x_{1},\dots,x_{n},z)=(y_{1},\dots,y_{n},z)$, so the formula in \cref{fact:per-double-expansion} can be interpreted as a quadratic polynomial in $x_{1},\dots,x_{n}$ (after conditioning on the value of $z$). In order to apply \cref{fact:non-degenerate} to this polynomial, we need this polynomial to have a multilinear degree-2 coefficient with large absolute value. For this, it suffices that $\per M[A,B]^{(i,j)}+\per M[A,B]^{(j,i)}$ has large absolute value for some $i\ne j$ (with $i,j\in A\cap B$). The following lemma will be useful for ensuring this condition.

\begin{lem}
\label{lem:per-no-cancel}Let $M$ be an $n\times n$ matrix and let $A,B\su \{1,\dots,n\}$ be subsets with $|A|=|B|$ such that $M[A,B]$
is $\lambda$-heavy. Suppose we are given an element $a\in B\setminus A$ and distinct elements $b_{1},b_{2}\in A\setminus B$. Then there are distinct $i,j\in\{ a,b_{1},b_{2}\} $
such that
\[
\left|\per M[A',B']^{(i,j)}+\per M[A',B']^{(j,i)}\right|\ge\lambda/2,
\]
where $A'=A\cup\{ a\} $ and $B'=(B\setminus\{ a\})\cup\{ i,j\} $.
\end{lem}

\begin{proof}
Suppose without loss of generality that $\per M[A,B]\ge\lambda$.
If we have
\[\per M[(A\setminus\{ b_{s}\} )\cup\{ a\} ,(B\setminus\{ a\} )\cup\{ b_{s}\} ]\ge-\lambda/2\]
for some $s\in\{ 1,2\}$, then we can take $i=b_{s}$ and
$j=a$. Indeed, then we have $A'=A\cup\{a\}$ and $B'=B\cup \{b_s\}$, and obtain $\per M[A',B']^{(i,j)}+\per M[A',B']^{(j,i)}\geq (-\lambda/2)+\lambda=\lambda/2$.

Otherwise, if there is no such $s\in\{ 1,2\}$, we can take $i=b_{1}$ and $j=b_{2}$. Then we have $A'=A\cup\{a\}$ and $B'=(B\setminus\{ a\})\cup \{b_1,b_2\}$, and obtain $\per M[A',B']^{(i,j)}+\per M[A',B']^{(j,i)}< (-\lambda/2)+(-\lambda/2)=-\lambda$.
\end{proof}

We end this section with two simple lemmas that illustrate how to apply \cref{fact:per-expansion,fact:per-double-expansion,lem:per-no-cancel} to ``grow'' heavy minors in a random symmetric matrix. Recall that $M_n$ is a random symmetric matrix, and that $M_{n-1}$ contains the first $n-1$ rows and columns of $M_n$.

\begin{lem}
\label{lem:simple-augment}Consider $A,B\su \{1,\dots,n-1\}$ with $\vert A\vert=\vert B\vert$, and fix any nonempty $I\subseteq\{ 1,\dots,n-1\} \setminus B$. Consider any outcome $M$ of $M_{n-1}$ such that $M[A,B]$ is $\lambda$-heavy, for some $\lambda>0$. Then
\[\Pr\left(\text{$M_{n}[A\cup\{ n\}, B\cup\{i\}]$ is $\lambda$-heavy for some $i\in I$}\,\big\vert \,M_{n-1}=M\right)\ge 1-2^{-|I|}.\]
\end{lem}

\begin{proof}
We condition on $M_{n-1}=M$. Let $x_{1},\dots,x_{n}$ be the entries in the last row of $M_{n}$.
Let us also condition on any outcome of the variables $x_{b}$ for $b\in B$. Now, by \cref{fact:per-expansion}, for each $i\in I$ we have
\[\per M_{n}[A\cup\{ n\} ,B\cup\{i\}] =x_i\per M_{n-1}[A ,B]+\sum_{b\in B}x_b\per M_{n-1}[A ,(B\cup\{i\})\setminus\{b\}].\]
Since $\vert \per M_{n-1}[A ,B]\vert\geq \lambda$, each $i\in I$ satisfies the desired condition $\vert\per M_{n}[A\cup\{ n\} ,B\cup\{i\}]\vert \geq \lambda$ with probability at least $1/2$, and (by our conditioning on the variables $x_{b}$ for $b\in B$) this happens independently for all $i\in I$.
\end{proof}

\begin{lem}
\label{lem:corner-might-work}Consider $A,B\su \{1,\dots,n-1\}$ with $\vert A\vert=\vert B\vert$, and consider an outcome $M$ of $M_{n-1}$ such that $M[A,B]$ is $\lambda$-heavy, for some $\lambda>0$. Then for any $a\in B\setminus A$ and any distinct $b_{1},b_{2}\in A\setminus B$, 
we can choose distinct $i,j\in\{ a,b_{1},b_{2}\}$ such that
\[
\Pr\left(M_{n}[A\cup\{ a,n\} ,(B\setminus\{ a\} )\cup\{ i,j,n\}]\text{ is }(\lambda/2)\text{-heavy}\,\big\vert\,M_{n-1}=M\right)\ge\frac14.
\]
\end{lem}

\begin{proof}Let us condition on $M_{n-1}=M$. By \cref{lem:per-no-cancel}, we can choose distinct $i,j\in\{ a,b_{1},b_{2}\}$ such that
\begin{equation}\label{eq-proof-corner-might-work}
\left|\per M_{n-1}[A',B']^{(i,j)}+\per M_{n-1}[A',B']^{(j,i)}\right|\ge\lambda/2,
\end{equation}
where $A'=A\cup\{ a\} $ and $B'=(B\setminus\{ a\} )\cup\{ i,j\}$. Note that $\{i,j\}\su \{ a,b_{1},b_{2}\}\su A'$ and that clearly $\{ i,j\}\su B'$.

Now, $\per M_{n}[A\cup\{ a,n\} ,(B\setminus\{ a\} )\cup\{ i,j,n\}]=\per M_n[A'\cup\{n\}, B'\cup\{n\}]$, so it suffices to show that with probability at least $1/4$ we have $\vert \per M_n[A'\cup\{n\}, B'\cup\{n\}]\vert\geq \lambda/2$.

Let $(x_1,\dots,x_{n-1},z)$ be the random entries of the last row (and the last column) of $M_n$. By \cref{fact:per-double-expansion}, we have
\[\per M_n[A'\cup\{n\}, B'\cup\{n\}]=z\per M_{n-1}[A',B']+\sum_{k\in A', \ell\in B'}x_kx_{\ell}\per M_{n-1}[A',B']^{(k,\ell)}.\]
Note that this is a quadratic polynomial in the variables $x_1,\dots,x_{n-1},z$, and the coefficient of $x_ix_j$ is precisely $\per M_{n-1}[A',B']^{(i,j)}+\per M_{n-1}[A',B']^{(j,i)}$. Recalling $i\neq j$ and \cref{eq-proof-corner-might-work}, \cref{fact:non-degenerate} now implies that $\Pr(\vert \per M_n[A'\cup\{n\}, B'\cup\{n\}]\vert< \lambda/2)\leq 3/4$.  This finishes the proof of \cref{lem:corner-might-work}.
\end{proof}

\section{Proof of \texorpdfstring{\cref{lem:grow-single-minor}}{Lemma~\ref{lem:grow-single-minor}}: growing a single heavy submatrix}\label{sec:grow-single-minor}

In this section we prove \cref{lem:grow-single-minor}. Recall that $L=\lfloor (\log n)/10\rfloor$ and that $X,Y\su \{1,\dots,n\}$ are disjoint subsets with $\vert X\vert=L$ and $\vert Y\vert=3L$. By reordering the rows and columns, we can assume without loss of generality that $X=\{ 1,\dots,L\}$ and $Y=\{ n-3L+1,\dots,n\}$.

\cref{lem:grow-single-minor} will be a consequence of the following two lemmas. The first of these lemmas is itself a weaker version of \cref{lem:grow-single-minor} (it also produces a heavy submatrix, but with less control over where it lies, and not with dimensions as close to $n\times n$).

\begin{lem}
\label{lem:grow-single-minor-weak}For any fixed $0<\delta<1/16$, the following holds for all integers $n\in \NN$ that are sufficiently large with respect to $\delta$. Let $\lambda=n^{(1/2-8\delta)n}$ and suppose that $R\in \NN$ satisfies $\delta n\le R\le 2\delta n$. Then with probability at least $1-e^{-\delta^2 n}$ there is a subset $B\subseteq\{ 1,\dots,n\} $ of size $n-R$ such that $M_{n}[\{ R+1,\dots,n\} ,B]$ is $\lambda$-heavy.
\end{lem}

To prove \cref{lem:grow-single-minor-weak} we adapt an argument in Tao and Vu's work \cite{TV09}, simultaneously tracking the propagation and growth of many heavy submatrices.

Our second lemma takes a heavy submatrix of a certain form with dimensions reasonably close to $n\times n$, and produces a slightly less heavy submatrix with dimensions much closer to $n\times n$, whose row and column sets satisfy the desired conditions in \cref{lem:grow-single-minor} (recall that we are assuming that $X=\{ 1,\dots,L\}$ and $Y=\{ n-3L+1,\dots,n\}$). To be more precise, we actually start with a submatrix contained inside the $n'\times n'$ matrix $M_{n'}$ for some $n'$ slightly smaller than $n$, and, conditioning on the outcome of $M_{n'}$, we only use the randomness from the additional rows and columns exposed when extending $M_{n'}$ to $M_n$

\begin{lem}
\label{lem:grow-single-minor-end}There is an absolute constant $c>0$ such that the following holds for all sufficiently large integers $n\in \NN$. Consider $\lambda>0$ and integers $L$ and $R$ satisfying $(\log n)/20<L<L^2<R< (n-5L^2-3L)/9$, and let $n'=n-8R-5L^2-3L$ and $\lambda'=\lambda/2^{R-L}$. Condition
on an outcome of $M_{n'}$ for which there is a subset $B\subseteq\{ 1,\dots,n'\} $
of size $n'-R$ such that $M_{n'}[\{ R+1,\dots,n'\} ,B]$
is $\lambda$-heavy. Then with probability at least $1-(1/8)\cdot n^{-c}$,
there is a set $B'$ of size $n-L$ with $\{ 1,\dots,n-3L\} \subseteq B'\subseteq \{ 1,\dots,n\}$, such
that $M_{n}[\{ L+1,\dots,n\} ,B']$ is $\lambda'$-heavy.
\end{lem}

It is now easy to deduce \cref{lem:grow-single-minor} from the two lemmas above.

\begin{proof}[Proof of \cref{lem:grow-single-minor}]
Let $c>0$ be the constant in \cref{lem:grow-single-minor-end}. Recall that we are considering some $\eps>0$ and that we are assuming that $n$ is sufficiently large with respect to $\eps$. Then in particular $L=\lfloor (\log n)/10\rfloor> (\log n)/20$. As mentioned at the beginning of this section, we may assume that $X=\{ 1,\dots,L\}$ and $Y=\{ n-3L+1,\dots,n\}$.

Let $\delta=\eps/32$, and $R=\lceil \delta n\rceil $, and note that by our assumption that $n$ is large with respect to $\eps$ we have $L<L^2<R<(n-5L^2-3L)/9$. Now let $n'=n-8R-5L^2-3L\geq (1-9\delta)n$ and  $\lambda=(n')^{(1/2-8\delta)n'}\geq n^{(1/2-15\delta)n}$ (again recalling that we assume $n$ to be large with respect to $\eps$). Note that then $\delta n'\leq R\leq 2\delta n'$.

By \cref{lem:grow-single-minor-weak}, with probability at least $1-e^{-\delta^2 n'}\geq 1-(1/8)\cdot n^{-c}$ (for $n$ sufficiently large with respect to $\eps$) there is a subset $B\subseteq\{ 1,\dots,n'\} $ of size $n'-R$ such that $M_{n'}[\{ R+1,\dots,n'\} ,B]$ is $\lambda$-heavy. Then by \cref{lem:grow-single-minor-end} and our choice of $c$, with probability at least $1-(1/8)\cdot n^{-c}$ there is a set $B'$ of size $n-L$ with $\{ 1,\dots,n-3L\} \subseteq B'\subseteq \{ 1,\dots,n\}$ such that $M_{n}[\{ L+1,\dots,n\} ,B']$ is $\lambda'$-heavy, where $\lambda'=\lambda/2^{R-L}\ge n^{(1/2-15\delta)n}/2^{\delta n}\ge n^{(1/2-16\delta)n}=n^{(1-\eps)n/2}$. Thus, the total probability that such a set $B'$ exists is at least $1-(1/4)\cdot n^{-c}$, as desired.
\end{proof}

\cref{lem:grow-single-minor-weak} will be proved in \cref{sec:grow-single-minor-weak} and \cref{lem:grow-single-minor-end} will be proved in \cref{sec:grow-single-minor-end}.

\subsection{Proof of \texorpdfstring{\cref{lem:grow-single-minor-weak}}{Lemma~\ref{lem:grow-single-minor-weak}}: propagation of heavy submatrices}\label{sec:grow-single-minor-weak}

In this subsection we prove \cref{lem:grow-single-minor-weak}, adapting an argument from \cite{TV09} to simultaneously track the propagation and growth of heavy submatrices as we expose more rows and columns of our random matrix. Roughly speaking, at each step we track submatrices of a certain form (with dimension growing by 1 at each step). At a given step, if we are guaranteed that many of our submatrices under consideration are heavy, then it is extremely likely that at the next step there will also be reasonably many heavy submatrices of the desired form. Moreover, depending on the structure of our random matrix at this step, one of the following is true, which will likely improve our situation in one of two ways. Either we have a good chance to dramatically increase the number of heavy submatrices in the next step, or we have a (very) good chance to have many submatrices in the next step which are much heavier than before. \cref{lem:grow-large-minors} below makes this precise.

After having proved \cref{lem:grow-large-minors}, we will deduce \cref{lem:grow-single-minor-weak} by iteratively applying \cref{lem:grow-large-minors}, adding a new row and a new column to our random matrix at every step. Most likely, there will be many steps where our situation improves in one of the two ways described above. However, there is an upper bound for the number of heavy submatrices that we can have at the end of the process (simply by counting the total number of submatrices of the form that we consider). Hence the first type of improvement, which significantly increases the number of heavy submatrices, cannot occur too many times. So, among the two ways we can ``improve the situation'', the second type of improvement must happen most of the time. This means that during our process we get submatrices that are more and more heavy, and at the end we find a reasonably large number of very heavy submatrices in our final matrix $M_n$ (in fact, we only need one such very heavy submatrix).

\begin{lem}
\label{lem:grow-large-minors}Fix $R, n\in \NN$. For $k\in \NN$ and real numbers $N>0$ and $\lambda>0$, let $E(k,N,\lambda)$ denote the event that there are at least $N$ different subsets $B\subseteq\{ 1,\dots,k+R\} $ with $\vert B\vert=k$ such that the matrix $M_{k+R}[\{ R+1,\dots,k+R\} ,B]$ is $\lambda$-heavy.

Then for any $k\in \NN$ with $k+R\le n$, and any real numbers $0<\delta<1/2$ as well as $K>1$, $\lambda>0$ and $N>0$, there is a partition $E(k,N,\lambda)=E'(k,N,\lambda)\cup E''(k,N,\lambda)$ of the event $E(k,N,\lambda)$ such that the following holds. Let $N^{+}=RN/(8K)$, $N^{-}=RN/(8n)$ and $\lambda^{+}=K^{1/2-\delta}\lambda$, and let $M,M',M''$ be any possible outcomes of $M_{k+R}$ satisfying $E(k,N,\lambda)$, $E'(k,N,\lambda)$ and $E''(k,N,\lambda)$ respectively. Then
\begin{align}
\Pr\left(E(k+1,N^{-},\lambda)\cond M_{k+R}=M\right)  &\ge 1-2e^{-R/8}.\label{eq:survive}\\
\Pr\left(E(k+1,N^{+},\lambda)\cond M_{k+R}=M'\right)  &\ge1/3.\label{eq:breed}\\
\Pr\left(E(k+1,N^{-},\lambda^{+})\cond M_{k+R}=M''\right) &\ge 1-4K^{-\delta}.
\label{eq:grow}
\end{align}
\end{lem}

\begin{proof}
We may assume without loss of generality that $N>0$ is an integer (indeed, otherwise we can replace $N$ by $\ceil{N}$, noting that the statement for $\ceil{N}$ implies the statement for $N$).

Let $x_{1},\dots,x_{k+R+1}$ be the entries in the last row of $M_{k+R+1}$.
For subsets $B\subseteq B'\subseteq\{1,\dots,k+R\}$ with sizes $k$ and $k+1$ respectively, we say that $B$ is a \emph{parent} of $B'$ and that $B'$ is a \emph{child} of $B$. For a subset $B'\subseteq\{1,\dots,k+R\}$ of size $k+1$, note that by \cref{fact:per-expansion} we have \begin{equation}\label{eq-expansion-parents}
    \per M_{k+R+1}[\{ R+1,\dots,k+R+1\} ,B']=\sum_B \per M_{k+R}[\{ R+1,\dots,k+R\} ,B]\cdot x_{B'\setminus B},
\end{equation}
where the sum is over all parents $B$ of $B'$ (here, with slight abuse of notation we write $x_{\{i\}}$ instead of $x_i$ for $i\in \{1,\dots,k+R\}$).

For each outcome of $M_{k+R}$ such that $E(k,N,\lambda)$ holds, let us fix subsets $B_{1},\dots,B_{N}$ as in the definition of $E(k,N,\lambda)$. Note that we always have $\vert \per M_{k+R}[\{ R+1,\dots,k+R\} ,B_i]\vert\geq \lambda$ for $i=1,\dots,N$.

Furthermore, for each outcome of $M_{k+R}$ satisfying $E(k,N,\lambda)$, let $S_{q}$ denote the collection of subsets
of $\{ 1,\dots,k+R\} $ of size $k+1$ which have exactly $q$ parents
among the sets $B_{1},\dots,B_{N}$, and let $S=S_{1}\cup\dots\cup S_{n}$ be
the collection of all such subsets which have at least one parent among $B_{1},\dots,B_{N}$. Furthermore, let $S_{\geq K}=S_{\lceil K\rceil }\cup\dots\cup S_{n}$ be the collection of all such subsets which have at least $K$ parents among $B_{1},\dots,B_{N}$. We say that $B'\in S$ is \emph{$\lambda'$}-heavy for some $\lambda'>0$ if $M_{k+R+1}[\{ R+1,\dots,k+R+1\} ,B']$
is $\lambda'$-heavy.

Since each of the sets $B_{1},\dots,B_N$ is a parent of exactly $R$ different sets $B'\in S$,
a double-counting argument shows that we have
\[
\sum_{q=1}^{n}q|S_{q}|= RN
\]
for each outcome of $M_{k+R}$ such that $E(k,N,\lambda)$ holds.

Now, let $E'(k,N,\lambda)\subseteq E(k, N,\lambda)$ be the event that $\sum_{q< K}q|S_{q}|\ge RN/2$, and condition on any outcome $M'$ of $M_{k+R}$ satisfying $E'(k,N,\lambda)$. Note that we then have $|S|\ge \sum_{q< K}|S_{q}|> RN/(2K)$. Furthermore note that for each $B'\in S$, at least one of the terms $\per M_{k+R}[\{ R+1,\dots,k+R\} ,B]$ on the left-hand side of \cref{eq-expansion-parents} has absolute value at least $\lambda$ (since $B'$ has at least one parent among $B_1,\dots,B_N$). Hence, by \cref{fact:non-degenerate-linear}, each $B'\in S$ is $\lambda$-heavy with probability at least $1/2$, and \cref{eq:breed} follows from Markov's inequality (to be precise, it follows from \cref{lem:Markov} applied with $p=1/2$ and $q=1/4$).

On the other hand, let $E''(k,N,\lambda)=E(k, N,\lambda)\setminus E'(k,N,\lambda)$ be the complementary event to $E'(k,N,\lambda)$ within $E(k, N,\lambda)$, i.e.\ the event that $\sum_{q\geq K}q|S_{q}|> RN/2$. Condition on any outcome $M''$ of $M_{k+R}$ satisfying $E''(k,N,\lambda)$, and note that then $|S_{\geq K}|\ge RN/(2n)$. Also note that for each $B'\in S_{\geq K}$, at least $K$ of the terms $\per M_{k+R}[\{ R+1,\dots,k+R\} ,B]$ on the left-hand side of \cref{eq-expansion-parents} have absolute value at least $\lambda$. Hence, by the Erd\H os--Littlewood--Offord inequality (specifically, \cref{lem:LO}, applied with $m=K$, $r=\lambda$ and $t=K^{1/2-\delta}$), each $B'\in S_{\geq K}$ is $\lambda^{+}$-heavy with probability at least $1-3K^{-\delta}$. Then, \cref{eq:grow} follows from Markov's inequality (specifically, we apply \cref{lem:Markov} with $p=1-3K^{-\delta}$ and $q=1/4$).

Finally, to prove \cref{eq:survive}, let us condition on any outcome $M$ of $M_{k+R}$ satisfying $E(k,N,\lambda)$. Say that for $i=1,\dots,N$, the set $B_{i}$ is \emph{good} if at least $R/4$ of its $R$ children
$B'\in S$ are $\lambda$-heavy. We claim that each $B_i$ is good with probability at least $1-e^{-R/8}$. Indeed, consider some fixed $i\in \{1,\dots,N\}$, and condition on any outcome of
the variables $x_{b}$ for $b\in B_{i}$. Now for each child $B'\in S$ of $B_i$, the sum in \cref{eq-expansion-parents} depends only on the outcome of $x_{B'\setminus B_i}$ (since for all other elements of $b\in B'$ the corresponding variable $x_b$ has already been fixed). Since $\vert \per M_{k+R}[\{R+1,\dots,k+R\},B_i]\vert\geq \lambda$, each child $B'\in S$ of $B_i$ is $\lambda$-heavy with probability at least $1/2$, independently for all children $B'\in S$. So, by the Chernoff bound (\cref{lem:Chernoff}), the set $B_{i}$ is indeed good with probability at least $1-e^{-2(R/4)^2/R}=1-e^{-R/8}$, as claimed.

Now, by Markov's inequality (specifically, \cref{lem:Markov}, applied with $p=1-e^{-R/8}$ and $q=1/2$), with probability at least $1-2e^{-R/8}$ at least $N/2$ of the sets $B_1,\dots,B_N$ are good. Whenever this is the case, there are at least $(N/2)\cdot (R/4)/n=RN/(8n)$ different $\lambda$-heavy sets $B'\in S$ (since each such set $B'\in S$ is a child of at most $k+1\leq n$ different sets $B_i$). This proves \cref{eq:survive}.
\end{proof}

Now we deduce \cref{lem:grow-single-minor-weak}.

\begin{proof}[Proof of \cref{lem:grow-single-minor-weak}] As in the lemma statement, let $0<\delta<1/16$ and assume that $n\in \NN$ is sufficiently large with respect to $\delta$ (sufficiently large to satisfy certain inequalities later in the proof). Let $R\in \NN$ be an integer satisfying $\delta n\leq R\leq 2\delta n$, and let $K=n^{1-\delta}$. Furthermore, recall the notation from the statement of \cref{lem:grow-large-minors}. We define random sequences $N_{1},\dots,N_{n-R}$
and $\lambda_{1},\dots,\lambda_{n-R}$ of positive real numbers by an iterative process. Let $N_{1}=\lambda_{1}=1$ and for each $1\leq k\leq n-R-1$ define $N_{k+1}$ and $\lambda_{k+1}$ as follows:
\begin{itemize}
\item[(i)] if $E'(k,N_k,\lambda_{k})$ and $E(k+1,N_{k}^{+},\lambda_{k})$
both hold, then let $N_{k+1}=N_{k}^{+}$ and $\lambda_{k+1}=\lambda_{k}$;
\item[(ii)] if $E''(k,N_k,\lambda_{k})$ and $E(k+1,N_{k}^{-},\lambda_{k}^{+})$
both hold, then let $N_{k+1}=N_{k}^{-}$ and $\lambda_{k+1}=\lambda_{k}^{+}$;
\item[(iii)] if neither (i) nor (ii) holds, but $E(k,N_{k},\lambda_{k})$ and $E(k+1,N_{k}^{-},\lambda_{k})$ both hold, then let $N_{k+1}=N_{k}^{-}$ and $\lambda_{k+1}=\lambda_{k}$;
\item[(iv)] otherwise, abort (and then our sequences are not well-defined).
\end{itemize}

Note that the event $E(1,N_1,\lambda_1)$ always holds. If we do not abort at any point in the above process, then $E(k,N_k,\lambda_k)$ holds for each $k$, and in particular there is a subset $B\subseteq\{1,\dots,n\}$ of size $\vert B\vert=n-R$ such that $M_n[\{R+1,\dots,n\},B]$ is $\lambda_{n-R}$-heavy. Thus, in order for the desired event in \cref{lem:grow-single-minor-weak} to hold, it is sufficient that the process does not abort and that $\lambda_{n-R}\ge n^{(1/2-8\delta)n}$. We will show that this happens with probability at least $1-e^{-\delta^2 n}$.

The main observation is that case (i) cannot occur too many times, simply because it is not possible for $N_k$ to ever be larger than $2^n$. Roughly speaking, it will follow from this observation and \cref{eq:breed} that $E'(k,N_k,\lambda_{k})$ is unlikely to occur too many times. This will then imply that case (ii) is likely to occur many times, meaning that $\lambda_{n-R}$ is large.

\begin{claim}\label{claim:dangerous-case}
Case (i) in the above process occurs for fewer than $\delta n$ different values of $k$.
\end{claim}
\begin{proof}
Note that whenever (i) holds, we have $N_{k+1}/N_k=N_k^+/N_k=R/(8K)\geq \delta n/8n^{1-\delta}=n^\delta/8$. On the other hand, whenever (ii) or (iii) holds, we have $N_{k+1}/N_k=N_k^-/N_k=R/(8n)\geq \delta n/8n=\delta/8$. Now suppose for the purpose of contradiction that (i) holds for at least $\delta n$ different $k$, and let us define $m=k+1$ for the last such value $k$. Note that then  we have
\[N_{m}\geq (n^\delta/8)^{\delta n}\cdot (\delta/8)^{m-\delta n}\geq n^{\delta^2 n}\cdot (\delta/8)^m\geq n^{\delta^2 n}\cdot (\delta/8)^n>2^n\]
for sufficiently large $n$. On the other hand, by our choice of $m$, case (i) holds for $k=m-1$, and so in particular the event $E(k+1,N_k^+,\lambda_k)=E(m,N_m,\lambda_m)$ holds. This means that there are at least $N_m>2^n$ different subsets $B\subseteq \{1,\dots,m+R\}\subseteq \{1,\dots,n\}$ satisfying the conditions in the definition of the event $E(m,N_m,\lambda_m)$. But this is clearly a contradiction, since the total number of subsets of $\{1,\dots,n\}$ is only $2^n$.
\end{proof}

The next observation is that if case (ii) occurs many times, then we are done.

\begin{claim}
If we do not abort at any point, and case (ii) occurs for at least $n-12\delta n$ different values of $k$, then $\lambda_{n-R}\ge n^{(1/2-8\delta)n}$.
\end{claim}
\begin{proof}
Whenever (ii) holds, we have $\lambda_{k+1}/\lambda_k=\lambda_k^+/\lambda_k=K^{1/2-\delta}$. On the other hand, whenever (i) or (iii) holds, we have $\lambda_{k+1}=\lambda_k$. So, if case (ii) occurs for at least $n-12\delta n$ values of $k$, then
\[\lambda_{n-R}\geq (K^{1/2-\delta})^{n-12\delta n} \geq n^{(1-\delta)\cdot (1/2-\delta)\cdot (n-12\delta n)}\geq n^{(1/2-2\delta)\cdot (n-12\delta n)}\geq n^{(1/2-8\delta)n}.\tag*{\qedhere}\]
\end{proof}

It now suffices to show that with probability at least $1-e^{-\delta^2 n}$, we do not abort and case (ii) occurs at least $n-12\delta n$ times. To this end, we define an auxiliary random process $W_1,\dots,W_{n-R}$ that evolves in parallel with $N_{1},\dots,N_{n-R}$ and $\lambda_{1},\dots,\lambda_{n-R}$. Namely, let $W_1=0$, and for $1\le k\le n-R-1$ let
\[W_{k+1}=W_k+(1-\delta)-\begin{cases}
3&\text{in case (i),}\\
1&\text{in case (ii),}\\
0&\text{in case (iii) or (iv).}
\end{cases}\]
Furthermore, if case (iv) occurs then let $W_{k+2}=W_{k+3}=\dots=W_{n-R}$ all be equal to the value of $W_{k+1}$ just defined (that is to say, we ``freeze'' the value of $W_k$ after the process aborts).

Note that $W_1,\dots,W_k$ are fully determined by the random matrix $M_{k+R}$ (which also determines its submatrices $M_{R+1},\dots,M_{k+R}$). Moreover, this defines a supermartingale, in the sense that $\E[W_{k+1}|M_{k+R}]\le W_k$ for each $k$ (provided $n$ is sufficiently large). To see this, consider any outcome of $M_k$ for which we have not yet aborted (meaning in particular that the event $E(k,N_k,\lambda_k)=E'(k,N_k,\lambda_k)\cup E''(k,N_k,\lambda_k)$ holds). If $E'(k,N_k,\lambda_k)$ holds, then $\E[W_{k+1}-W_{k}|M_k]\le  (1-\delta)-(1/3)\cdot 3\le-\delta$ by \cref{eq:breed}. On the other hand, if $E''(k,N_k,\lambda_k)$ holds, then $\E[W_{k+1}-W_{k}|M_k]\le (1-\delta)-(1-4K^{-\delta})=-\delta+4K^{-\delta}\le 0$ for sufficiently large $n$, by \cref{eq:grow}. In addition, observe that $|W_i-W_{i-1}|\le 3$ for each $1<i\le n-R$.

By \cref{lem:AH2} (with $Z_k=M_{k+R}$ for $k=1,\dots,n-R$, and $c=3$) we have $W_{n-R}\le 5\delta n$ with probability at least $1-e^{-(25/18)\delta^2 n}\ge 1-(1/2)e^{-\delta^2 n}$. Also, by \cref{eq:survive} and the union bound, the probability that we ever abort is bounded by $(n-R)\cdot 2e^{-R/8}\leq n\cdot 2e^{-\delta n/8}\leq (1/2)e^{-\delta^2 n}$. But note that if we never abort, then
\[W_{n-R}=(n-R-1)(1-\delta)-3X_{\text{(i)}}-X_{\text{(ii)}},\]
where $X_{\text{(i)}}$ is the number of times that case (i) occurs, and $X_{\text{(ii)}}$ is the number of times that case (ii) occurs. Recall that $X_{\text{(i)}}\le \delta n$ by \cref{claim:dangerous-case}. Hence, if $W_{n-R}\le 5\delta n$ and the process does not abort, then case (ii) occurs $X_{\text{(ii)}}\ge (n-R-1)(1-\delta)-3\delta n-5\delta n\ge n-12\delta n$ times, which by \cref{claim:dangerous-case} implies that $\lambda_{n-R}\geq n^{(1/2-8\delta)n}$. Thus, we have indeed shown that with probability at least $1-e^{-\delta^2 n}$ the process does not abort and we have $\lambda_{n-R}\geq n^{(1/2-8\delta)n}$.
\end{proof}

\subsection{Proof of \texorpdfstring{\cref{lem:grow-single-minor-end}: ``filling out''}{Lemma~\ref{lem:grow-single-minor-end}: "filling out"} a single heavy submatrix}\label{sec:grow-single-minor-end}

In this subsection we prove \cref{lem:grow-single-minor-end}. It will be a consequence
of the following two lemmas, which (in two slightly different ways) ``grow'' a heavy submatrix by exposing a few additional rows and columns.

\begin{lem}
\label{lem:iterative-cover}Let $1\leq S<n$ and $\lambda>0$, and condition on an outcome of $M_{n}$ for which there is a subset $B\subseteq \{1,\dots,n\}$ of size $n-S$ such that $M_{n}[\{ S+1,\dots,n\} ,B]$ is $\lambda$-heavy.
Then with probability at least $1-3S\cdot 2^{-S}-e^{-S/6}$, there is a set $B'$ of size $n+2S$ with $\{ 1,\dots,n\}\subseteq B'\subseteq \{ 1,\dots,n+3S\}$ such that the matrix $M_{n+3S}[\{ S+1,\dots,n+3S\} ,B']$ is $\lambda$-heavy.
\end{lem}

\begin{lem}
\label{lem:iterative-growth}Let $2\leq T<S<n$ and $\lambda>0$, and condition on an outcome of $M_{n}$ for which there is a set $B$ of size $n-S$ with $\{1,\dots,S\}\subseteq B\subseteq \{1,\dots,n\}$ such that $M_{n}[\{ S+1,\dots,n\} ,B]$ is $\lambda$-heavy. Then with probability at least $1-5S\cdot 2^{-T}-e^{-S/40}$, there is a set $B'$ of size $n+5S-T$ with $\{1,\dots,T\}\subseteq B\subseteq \{1,\dots,n+5S\}$ such that $M_{n+5S}[\{ T+1,\dots,n+5S\} ,B']$ is $\lambda/2^{S-T}$-heavy.
\end{lem}

Before proving \cref{lem:iterative-cover,lem:iterative-growth}, we deduce \cref{lem:grow-single-minor-end}.

\begin{proof}[Proof of \cref{lem:grow-single-minor-end}]
Recall that $n'=n-8R-5L^2-3L$, and that we are conditioning on an outcome of $M_{n'}$ for which there is a subset $B\subseteq\{ 1,\dots,n'\} $ of size $n'-R$ such that $M_{n'}[\{ R+1,\dots,n'\} ,B]$ is $\lambda$-heavy.

First, by \cref{lem:iterative-cover} (applied with $S=R<n'$), with probability at least $1-3R\cdot 2^{-R}-e^{-R/6}$, there is a set $B_1$ of size $n'+2R$ with $\{ 1,\dots,R\}\subseteq \{ 1,\dots,n'\}\subseteq B'\subseteq \{ 1,\dots,n'+3R\}$ such that $M_{n'+3R}[\{ R+1,\dots,n'+3R\} ,B_1]$ is $\lambda$-heavy. Let us now condition on such an outcome for $M_{n'+3R}$.

Then, by \cref{lem:iterative-growth} (applied with $T=L^2$ and $S=R<n'+3R$), we obtain that with probability at least $1-5R\cdot 2^{-L^2}-e^{-R/40}$
there is a set $B_2$ of size $n'+8R-L^2$ with $\{1,\dots,L^2\}\subseteq B_2\subseteq \{1,\dots,n'+8R\}$ such that
that the matrix $M_{n'+8R}[\{ L^2+1,\dots,n'+8R\} ,B_2]$ is $(\lambda/2^{R-L^2})$-heavy. Let us condition on such an outcome for $M_{n'+8R}$.

Applying \cref{lem:iterative-growth} again (this time with $T=L$ and $S=L^2<n'+8R$), we now get that with probability at least $1-5L^2\cdot 2^{-L}-e^{-L^2/40}$, there is a set $B_3$ of size $n'+8R+5L^2-L$ with $\{1,\dots,L\}\subseteq B_3\subseteq \{1,\dots,n'+8R+5L^2\}$ such that
that $M_{n'+8R+5L^2}[\{ L+1,\dots,n'+8R+5L^2\} ,B_3]$ is $\lambda'$-heavy, where $\lambda'=(\lambda/2^{R-L^2})/2^{L^2-L}=\lambda/2^{R-L}$. Let us condition on such an outcome for $M_{n'+8R+5L^2}$

Finally, by \cref{lem:iterative-cover} (applied with $S=L<n'+8R+5L^2$), with probability at least $1-3L\cdot 2^{-L}-e^{-L/6}$, there is a set $B'$ of size $n'+8R+5L^2+2L=n-L$ with
\[\{ 1,\dots,n-3L\}=\{ 1,\dots,n'+8R+5L^2\}\subseteq B'\subseteq \{ 1,\dots,n'+8R+5L^2+3L\}=\{ 1,\dots,n\}\]
such that $M_{n}[\{ L+1,\dots,n\} ,B']$ is $\lambda'$-heavy.

The probability that all four steps succeed is at least
\[1-(3R\cdot 2^{-R}+e^{-R/6})-(5R\cdot 2^{-L^2}+e^{-R/40})-(5L^2\cdot 2^{-L}+e^{-L^2/40})-(3L\cdot 2^{-L}+e^{-L/6})\ge 1-(1/8)\cdot n^{-c}\]
for some small constant $c>0$ (recall that $(\log n)/20<L<L^2<R<n$ and that $n$ is sufficiently large).
\end{proof}

We now prove \cref{lem:iterative-cover}.

\begin{proof}[Proof of \cref{lem:iterative-cover}]For any $m\in \NN$ with $n\leq m\leq n+3S$, define the random variable $Q_{m}$ to be the minimum value of $|\{ 1,\dots,n\} \setminus B'|$ among all subsets $B'\subseteq\{1,\dots,m\}$ of size $m-S$ such that $M_{m}[\{ S+1,\dots,m\} ,B']$
is $\lambda$-heavy. If no such subset $B'$ exists, let $Q_{m}=\infty$.

Recall that in \cref{lem:iterative-cover} we are conditioning on an outcome of $M_n$ for which there is a subset $B\subseteq \{1,\dots,n\}$ of size $n-S$ such that $M_{n}[\{ S+1,\dots,n\} ,B]$ is $\lambda$-heavy. This means that $Q_n\leq |\{ 1,\dots,n\} \setminus B|= S$.

For $n<m\le n+3S$, we say that step $m$ is a \emph{failure} if $Q_{m}>Q_{m-1}$. We say that step $m$ is \emph{progress} if $Q_{m}<Q_{m-1}$ or if $Q_{m}=Q_{m-1}=0$
or if $Q_{m-1}=\infty$.

For any $n<m\le n+3S$, when conditioning on any outcome of $M_{m-1}$, we claim that step $m$ is a failure with probability at most $2^{-S}$. Indeed, if $Q_{m-1}<\infty$, let $B'\subseteq\{1,\dots,m-1\}$ be a subset of size $m-1-S$ with $|\{ 1,\dots,n\} \setminus B'|=Q_{m-1}$ such that $M_{m-1}[\{ S+1,\dots,m-1\} ,B']$ is $\lambda$-heavy. By applying \cref{lem:simple-augment} with $I=\{ 1,\dots,m-1\} \setminus B'$, we see that with probability at least $1-2^{-S}$ there exists some $i\in \{ 1,\dots,m-1\} \setminus B'$ such that $M_{m}[\{ S+1,\dots,m\}, B'\cup\{i\}]$ is $\lambda$-heavy (which in particular implies $Q_m\leq |\{ 1,\dots,n\} \setminus (B'\cup\{i\})|\leq Q_{m-1}$). On the other hand, if $Q_{m-1}=\infty$, then step $m$ cannot be a failure. So this indeed shows that in any case (when conditioning on any outcome of $M_{m-1}$), step $m$ is a failure with probability at most $2^{-S}$.

Furthermore, for any $n<m\le n+3S$, when conditioning on any outcome of $M_{m-1}$, we claim that step $m$ is progress with probability at least $1/2$. Indeed, if $Q_{m-1}\notin\{ 0,\infty\}$, let $B'\subseteq\{1,\dots,m-1\}$ be a subset of size $m-1-S$ with $|\{ 1,\dots,n\} \setminus B'|=Q_{m-1}$ such that $M_{m-1}[\{ S+1,\dots,m-1\} ,B']$ is $\lambda$-heavy. We can then apply \cref{lem:simple-augment} with $I=\{ 1,\dots,n\} \setminus B'$, to see that with probability at least $1/2$ there exists some $i\in \{ 1,\dots,n\} \setminus B'$ such that $M_{m}[\{ S+1,\dots,m\}, B'\cup\{i\}]$ is $\lambda$-heavy (which in particular implies $Q_m\leq |\{ 1,\dots,n\} \setminus (B'\cup\{i\})|<Q_{m-1}$). If $Q_{m-1}=\infty$, by definition step $m$ is always progress. If $Q_{m-1}=0$, then step $m$ is progress if and only if it is not failure, and we already showed that it is failure with probability at most $2^{-S}\leq 1/2$. This shows that in any case (when conditioning on any outcome of $M_{m-1}$), step $m$ is progress with probability at least $1/2$.

Hence the number of progress steps among the $3S$ steps $m\in \{n+1,\dots,n+3S\}$ stochastically dominates a sum of $3S$ independent $\Ber(1/2)$ random variables. By the Chernoff bound (\cref{lem:Chernoff}) such a sum is at least $S$ with probability at least $1-e^{-2(S/2)^2/(3S)}=1-e^{-S/6}$. Thus, the number of progress steps is also at least $S$ with probability at least $1-e^{-S/6}$. Furthermore, note that by the union bound, with probability at least $1-3S\cdot 2^{-S}$ none of the $3S$ steps $m\in \{n+1,\dots,n+3S\}$ is a failure.

If there are no failures and at least $S$ progress steps, then we must have $Q_{n+3S}=0$. Hence, with probability at least $1-3S\cdot 2^{-S}-e^{-S/6}$ there exists a subset $B'\subseteq\{1,\dots,n+3S\}$ of size $n+2S$ such that $M_{n+3S}[\{ S+1,\dots,n+3S\} ,B']$ is $\lambda$-heavy and $|\{ 1,\dots,n\} \setminus B'|=0$ (meaning that $\{ 1,\dots,n\}\su B'$).
\end{proof}

The proof of \cref{lem:iterative-growth} follows a similar strategy as the proof of \cref{lem:iterative-cover} above. 

\begin{proof}[Proof of \cref{lem:iterative-growth}]
For any $m\in \NN$ with $n\leq m\leq n+5S$, define the random variable $Q_{m}$ to be the minimal number $Q\in \{T, T+1, T+2,\dots\}$ such that there is a set $B'$ of size $m-Q$ with $\{ 1,\dots,Q\}\su B'\su \{1,\dots,m\}$ such that the matrix $M_{m}[\{ Q+1,\dots,m\} ,B']$ is $(\lambda/2^{S-Q})$-heavy. If there is no $Q\in \{T, T+1, T+2,\dots\}$ for which such a set $B'$ exists, we define $Q_m=\infty$.

Recall that in \cref{lem:iterative-growth}, we are conditioning on an outcome of $M_{n}$ for which there is a set $B$ of size $n-S$ with $\{1,\dots,S\}\subseteq B\subseteq \{1,\dots,n\}$ such that
that $M_{n}[\{ S+1,\dots,n\} ,B]$ is $\lambda$-heavy. This means that $Q_n\leq S$ (recall that $S>T$).

For $n<m\le n+5S$, we say that step $m$ is a \emph{failure} if $Q_{m}>Q_{m-1}$. We say that step $m$ is \emph{progress} if $Q_{m}<Q_{m-1}$ or if $Q_{m}=Q_{m-1}=T$ or if $Q_{m-1}=\infty$.

For any $n<m\leq n+5S$, when conditioning on any outcome of $M_{m-1}$, we claim that step $m$ is a failure with probability at most $2^{-T}$. Indeed, if $Q_{m-1}<\infty$, let $B'$ be a set of size $m-1-Q_{m-1}$ with $\{ 1,\dots,Q_{m-1}\}\su B'\su \{1,\dots,m-1\}$ such that the matrix $M_{m-1}[\{ Q_{m-1}+1,\dots,m-1\} ,B']$ is $(\lambda/2^{S-Q_{m-1}})$-heavy. By applying \cref{lem:simple-augment} with $I=\{ 1,\dots,m-1\} \setminus B'$, we obtain that with probability at least $1-2^{-Q_{m-1}}\geq 1-2^{-T}$ there exists some $i\in \{ 1,\dots,m-1\} \setminus B'$ such that the matrix $M_{m}[\{ Q_{m-1}+1,\dots,m\} ,B'\cup \{i\}]$ is $(\lambda/2^{S-Q_{m-1}})$-heavy (which in particular implies that $Q_m\leq Q_{m-1}$). On the other hand, if $Q_{m-1}=\infty$, then step $m$ cannot be a failure.

We furthermore claim that for any $n<m\leq n+5S$, when conditioning on any outcome of $M_{m-1}$, step $m$ is progress with probability at least $1/4$. First assume that $Q_{m-1}\not\in \{T,\infty\}$, and let $B'$ be a set of size $m-1-Q_{m-1}$ with $\{ 1,\dots,Q_{m-1}\}\su B'\su \{1,\dots,m-1\}$ such that $M_{m-1}[\{ Q_{m-1}+1,\dots,m-1\} ,B']$ is $(\lambda/2^{S-Q_{m-1}})$-heavy. We can then apply \cref{lem:corner-might-work} to the sets $\{ Q_{m-1}+1,\dots,m-1\}$ and $B'$. Since $\{ 1,\dots,Q_{m-1}\}\su B'$ and $\vert B'\vert =m-1-Q_{m-1}$, we have $\vert \{ Q_{m-1}+1,\dots,m-1\}\setminus B'\vert=Q_{m-1}\geq T\geq 2$. So we can find two distinct elements $b_1,b_2\in \{ Q_{m-1}+1,\dots,m-1\}\setminus B'$. Let us furthermore take $a=Q_{m-1}\in B'\setminus \{ Q_{m-1}+1,\dots,m-1\}$. By \cref{lem:corner-might-work}, there exist distinct elements $i,j\in \{a,b_1,b_2\}$ such that the set $B^*=(B'\setminus\{a\})\cup \{i,j,m\}$ has the property that the matrix $M_m[\{ Q_{m-1},\dots,m\} ,B^*]$ is $(\lambda/2^{S-Q_{m-1}+1})$-heavy with probability at least $1/4$. Also note that $\vert B^*\vert=m-(Q_m-1)$ and $\{ 1,\dots,Q_{m-1}-1\}\su B^*\su \{1,\dots,m\}$. So we can conclude that with probability at least $1/4$ we have $Q_m\leq Q_{m-1}-1$, meaning that step $m$ is progress.

In order to finish proving the claim that for any $n<m\leq n+5S$ (when conditioning on any outcome of $M_{m-1}$), step $m$ is progress with probability at least $1/4$, it only remains to consider the cases $Q_{m-1}=\infty$ and $Q_{m-1}=T$. If $Q_{m-1}=\infty$, then step $m$ is always progress. If $Q_{m-1}=T$, then step $m$ is progress if and only if it is not failure, and we already proved that step $m$ is failure with probability at most $2^{-T}\leq 1/4\leq 3/4$.

Having proved this claim, we can now conclude that the number of progress steps among the $5S$ steps $m\in \{n+1,\dots,n+5S\}$ stochastically dominates a sum of $5S$ independent $\Ber(1/4)$ random variables. By the Chernoff bound (\cref{lem:Chernoff}) such a sum is at least $S$ with probability at least $1-e^{-2(S/4)^2/(5S)}=1-e^{-S/40}$. Thus, the number of progress steps is also at least $S$ with probability at least $1-e^{-S/40}$. Furthermore, note that by the union bound, with probability at least $1-5S\cdot 2^{-T}$ none of the $5S$ steps $m\in \{n+1,\dots,n+5S\}$ is a failure.

If there are no failures and at least $S\geq S-T$ progress steps, then we must have $Q_{n+5S}=T$. Hence, with probability at least $1-5S\cdot 2^{-T}-e^{-S/40}$ there is a set $B'$ of size $n+5S-T$ with $\{ 1,\dots,T\}\su B'\su \{1,\dots,n+5S\}$ such that the matrix $M_{n+5S}[\{ T+1,\dots,n+5S\} ,B']$ is $(\lambda/2^{S-T})$-heavy.
\end{proof}

\section{Proof of \texorpdfstring{\cref{lem:endgame-step}}{Lemma~\ref{lem:endgame-step}}: survival of heavy submatrices}\label{sec:endgame-step}

In this section we prove \cref{lem:endgame-step}. Recall that we call subsets $S_1,\dots,S_m$ of some ground set $S$ complement-disjoint, if their complements $S\setminus S_1,\dots,S\setminus S_m$ are disjoint (and note that this condition is in particular satisfied if $S_1=\dots=S_m=S$).

As in the lemma statement, let $\lambda>0$ and let $A_1,\dots,A_m,B_1,\dots,B_m$ be complement-disjoint subsets of $\{1,\dots,n\}$ of size $n-L$. Recall that we are conditioning on an outcome of the matrix $M_n$ such that we have $\vert \per M_n[A_\ell,B_\ell]\vert\geq \lambda$ for $\ell=1,\dots,m$. Also recall that we are assuming that $m$ is large.

Let $x_{1},\dots,x_{n},z$ be the entries of the new row and column in $M_{n+1}$, and let us condition on any fixed  outcome of $z$ (which we no longer view as being random).

First, starting from the complement-disjoint subsets $A_1,\dots,A_m,B_1,\dots,B_m\su \{1,\dots,n\}$ of size $n-L$, we will construct certain complement-disjoint subsets $A_1^*,\dots,A_m^*,B_1^*,\dots,B_m^*\su \{1,\dots,n\}$ of size $n-L+1$. The plan is then to choose the desired subsets $A_1',\dots,A_{m'}',B_1',\dots,B_{m'}'$ in \cref{lem:endgame-step} to each be of the form $A_i^*\cup \{n+1\}$ or $B_i^*\cup \{n+1\}$, for suitably chosen $i\in \{1,\dots,m\}$.

\begin{claim}\label{claim-sets-A-B-prime}
We can find quadruples $(A^*_\ell,B^*_\ell,i_\ell,j_\ell)$ for $\ell\in\{ 1,\dots,m\}$, satisfying the following conditions.
\begin{itemize}
\item For each $\ell\in\{ 1,\dots,m\}$, we have $A_{\ell}^*,B_{\ell}^*\subseteq\{ 1,\dots,n\} $ and $\vert A^*_\ell\vert=\vert B^*_\ell\vert=n-L+1$, and furthermore $i_\ell,j_\ell\in A_\ell^*\cap B_\ell^*$.
\item The elements $i_{1},j_{1}, \dots, i_{m},j_{m}\in \{1,\dots,n\}$ are distinct.
\item The sets $A_{1}^*,B_{1}^*,\dots,A_{m}^*,B_{m}^*$ are complement-disjoint (over the ground set $\{ 1,\dots,n\}$).
\item For each $\ell\in\{ 1,\dots,m\}$, if we view $\per M_{n+1}[A_{\ell}^*\cup\{ n+1\} ,B_{\ell}^*\cup\{ n+1\} ]$ as a polynomial in $x_{1},\dots,x_{n}$, then the coefficient of $x_{i_{\ell}}x_{j_{\ell}}$ has absolute value at least $\lambda/2$.
\end{itemize}
\end{claim}

\begin{proof}
First consider the case $L=1$. For every $\ell\in \{1,\dots,m\}$, let us take $A_{\ell}^*=B_{\ell}^*=\{ 1,\dots,n\}$, let $i_{\ell}$ be the single element of $\{ 1,\dots,n\}\setminus A_\ell$, and let $j_{\ell}$ be the single element of $\{ 1,\dots,n\}\setminus B_\ell$. Note that the second condition is satisfied by the complement-disjointness of the sets $A_{1},B_{1},\dots,A_{m},B_{m}$. For the last condition, observe that by \cref{fact:per-double-expansion} and the symmetry of our matrices, the coefficient of $x_{i_{\ell}}x_{j_{\ell}}$ in $\per M_{n+1}$ equals $\per M_n[A_{\ell},B_{\ell}]+\per M_n[B_{\ell},A_{\ell}]=2\per M_n[A_{\ell},B_{\ell}]$, which has absolute value at least $2\lambda\ge\lambda/2$.

Now, we consider the case $L\geq 2$. For each $\ell\in \{1,\dots,m\}$, choose $a_\ell\in B_\ell\setminus A_{\ell}$ and distinct $b_{\ell},b_{\ell}'\in A_\ell\setminus B_{\ell}$ (this is possible since $A_\ell$ and $B_\ell$ are complement-disjoint and have size at most $n-2$). Now let $i_{\ell},j_{\ell}\in \{a_\ell,b_\ell,b_\ell'\}$ be
as in \cref{lem:per-no-cancel}, and let $A_{\ell}^*=A_{\ell}\cup\{ a_\ell\} $ and $B_{\ell}^*=(B_{\ell}\setminus\{ a_\ell\} )\cup\{ i_{\ell},j_{\ell}\}$. For the last condition, note that by  \cref{fact:per-double-expansion} the coefficient of $x_{i_{\ell}}x_{j_{\ell}}$ in $\per M_{n+1}[A_{\ell}^*\cup\{ n+1\} ,B_{\ell}^*\cup\{ n+1\}]$ equals $\per M_n[A_{\ell}^*,B_{\ell}^*]^{(i_\ell,j_\ell)}+\per M_n[B_{\ell}^*,A_{\ell}^*]^{(j_\ell,i_\ell)}$, which has absolute value at least $\lambda/2$ by the conclusion of \cref{lem:per-no-cancel}.
\end{proof}

Fix quadruples $(A^*_\ell,B^*_\ell,i_\ell,j_\ell)$ for $\ell\in\{ 1,\dots,m\}$ as in \cref{claim-sets-A-B-prime}. Let $I=\{i_{1},j_{1}, \dots, i_{m},j_{m}\}\su \{1,\dots,n\}$, and let us condition on any outcome for all the variables $x_i$ with $i\notin I$ (which we will no longer view as being random). For every $\ell=1,\dots,m$, define $P_{\ell}=\per M_{n+1}[A_{\ell}^*\cup\{ n+1\} ,B_{\ell}^*\cup\{ n+1\}]$, viewed as a polynomial in the variables $x_i$ for $i\in I$, but with all quadratic terms $x_i^2$ replaced by $1$ (recall that our variables $x_i$ are chosen in $\{-1,1\}$). Then $P_\ell$ is a multilinear quadratic polynomial and the coefficient of $x_{i_{\ell}}x_{j_{\ell}}$ has absolute value at least $\lambda/2$.

Now, after our conditioning, the only remaining randomness comes from the $2m$ variables $x_i$ for $i\in I$. It suffices to show that for sufficiently large $m$ we have
\begin{equation}\label{eq-many-l-satisfy}
\Pr\left(|P_{\ell}|\ge\lambda/(4n^4)\text{ for at least }m/36\text{ indices }\ell\in \{1,\dots,m\}\right)\geq 1-m^{-1/24}.
\end{equation}
Indeed, if the event in \cref{eq-many-l-satisfy} holds, then we can take the sets $A_1',\dots,A_{m'}, B_1',\dots,B_{m'}$ in \cref{lem:endgame-step} to be the sets $A_\ell^*\cup\{n+1\}$ and $B_\ell^*\cup\{n+1\}$ for  $m'=\ceil{m/36}$ different indices $\ell\in \{1,\dots,n\}$ for which we have $|P_{\ell}|=\vert \per M_{n+1}[A_{\ell}^*\cup\{ n+1\} ,B_{\ell}^*\cup\{ n+1\}]\vert\ge\lambda/(4n^{4})$. 

Let $\sigma=\lambda/(4n^2)$ and $\tau=\lambda/(4n^4)$. Using the notation introduced above \cref{lem:MNV}, for each $\ell\in \{1,\dots,m\}$ we consider the graph $G_\ell=G^{(\tau)}(P_\ell)$ on the vertex set $I$ whose edges correspond to the coefficients of the polynomial $P_\ell$ of absolute value at least $\tau$. We say that the index $\ell\in \{1,\dots,m\}$ is \emph{easy} if the graph $G_\ell$ has matching number $\nu(G_\ell)\geq m^{1/6}$. If there are many easy indices, then we can prove \cref{eq-many-l-satisfy} using the Meka--Nguyen--Vu polynomial anti-concentration inequality (\cref{lem:MNV}), as follows.

\begin{claim}
If there are at least $m/3$ easy indices $\ell\in \{1,\dots,m\}$, then \cref{eq-many-l-satisfy} holds.
\end{claim}

\begin{proof}
Recall that for each easy index $\ell$, we have $\nu(G^{(\tau)}(P_\ell))=\nu(G_\ell)\geq m^{1/6}$. Hence by \cref{lem:MNV} we have that
\[\Pr(\vert P_\ell\vert\geq \tau)\geq 1-\frac{(\log \nu(G_\ell))^{C}}{\nu(G_\ell)^{1/2}}\geq 1-\nu(G_\ell)^{-1/3}\geq 1-m^{-1/18}\]
for sufficiently large $m$ (where $C$ is the absolute constant appearing in the statement of \cref{lem:MNV}). Hence by Markov's inequality (specifically \cref{lem:Markov} applied with $p=1-m^{-1/18}$ and $q=1/2$), with probability at least $1-2m^{-1/18}\geq 1-m^{-1/24}$ we have $\vert P_\ell\vert\geq \tau=\lambda/(4n^{4})$ for at least $(1/2)\cdot (m/3)=m/6$ easy indices $\ell\in \{1,\dots,m\}$ (again assuming that $m$ is sufficiently large). This in particular proves \cref{eq-many-l-satisfy}.
\end{proof}

Let us from now on assume that there are at least $2m/3$ indices $\ell\in \{1,\dots,m\}$ which are not easy. For each of these non-easy $\ell$, since $G_\ell$ has no large matching it must have a large vertex cover, as follows.

\begin{claim}\label{claim:vertex-cover}
For every non-easy index $\ell\in \{1,\dots,m\}$, there is a subset $S_\ell\su I$ of size $\vert S_\ell\vert \leq 2m^{1/6}$ such that each edge of the graph $G_\ell$ contains at least one vertex in $S_\ell$ (in other words, $S_\ell$ is a vertex cover of the graph $G_\ell$).
\end{claim}
\begin{proof}
Let us take a maximal collection of disjoint edges in $G_\ell$ (this collection consists of at most $\nu(G_\ell)< m^{1/6}$ edges), and let $S_\ell$ consist of all the vertices contained in one of these edges. Then by the maximality of the chosen edge collection, each edge of $G_\ell$ must contain at least one vertex in $S_\ell$.
\end{proof}

For each non-easy $\ell$, fix a subset $S_\ell\su I$ as in \cref{claim:vertex-cover}. We now briefly describe the strategy of the remainder of the proof. The idea is that all the degree-2 terms of $P_\ell$ whose coefficient has large absolute value must contain a variable whose index is in $S_\ell$. So if we condition on outcomes of $x_i$ for $i\in S_\ell$, then $P_\ell$ ``essentially'' becomes a linear polynomial (apart from some small terms that we can ignore). If this linear polynomial has many  coefficients with large absolute value, then we can apply the Erd\H os--Littlewood--Offord inequality to show that $|P_\ell|$ is typically quite large. However, it is possible that for most of the non-easy $\ell$, we end up with linear polynomials which have few coefficients of large absolute value (in which case we will not be able to use such an argument). It turns out that this is unlikely to happen unless for many $\ell$ the polynomial $P_\ell$ each essentially depend on only a few of the variables $x_i$ (we will call such indices $\ell$ \emph{short}). We will be able to handle the case that there are many such indices $\ell$ using the Azuma--Hoeffding inequality.

Let us say that a variable $x_i$ with $i\in I$ is \emph{bad} if we have $i\in S_\ell$ for at least $m^{1/3}$ non-easy indices $\ell\in \{1,\dots,m\}$. Note that by a simple counting argument, there are at most $m\cdot 2m^{1/6}/m^{1/3}=2m^{5/6}$ bad variables. We say that a variable $x_i$ with $i\in I$ is \emph{good} if it is not bad. Let $I_{\text{good}}\su I$ be the set of all $i\in I$ such that $x_i$ is good (and note that $\vert I_{\text{good}}\vert\leq \vert I\vert =2m$).

In addition to all of our previous conditioning, let us now also condition on any fixed outcome of all bad variables $x_i$. This means that at this point the only remaining randomness comes from the variables $x_i$ with $i\in I_{\text{good}}$. After fixing the outcomes for the bad variables, we can interpret each polynomial $P_\ell$ (for $\ell\in \{1,\dots,m\}$) as a polynomial in the variables $x_i$ with $i\in I_{\text{good}}$. It suffices to prove \cref{eq-many-l-satisfy} with this additional conditioning on the outcomes of the bad variables.

Note that for each $\ell\in \{1,\dots,m\}$ which is not easy, for any distinct $i,j\in I_{\text{good}}\setminus S_\ell$ the coefficient of $x_ix_j$ in $P_\ell$ has absolute value less than $\tau$. Indeed, this follows from the definition of the graph $G_\ell=G^{(\tau)}(P_\ell)$ and the fact that every edge of $G_{\ell}$ contains at least one vertex in $S_\ell$.

Recall that $\sigma=\lambda/(4n^2)$. We say that an index $\ell\in \{1,\dots,m\}$ is \emph{short} if there are at most $6m^{1/6}$ good variables $x_i$ such that $x_i$ appears in a term of $P_\ell$ whose coefficient has absolute value at least $\sigma$.

\begin{claim}
If there are at least $m/3$ short indices $\ell\in \{1,\dots,m\}$, then \cref{eq-many-l-satisfy} holds.
\end{claim}
\begin{proof}
Recall that we are viewing each $P_\ell$ as a polynomial in the variables $x_i$ with $i\in I_{\text{good}}$. Now, let $P_\ell'$ be the polynomial obtained from $P_\ell$ by deleting all terms whose coefficients have absolute value less than $\sigma$. Note that for each short index $\ell\in \{1,\dots,m\}$, the polynomial $P_\ell'$ contains at most $6m^{1/6}$ different variables $x_i$. Also note that we always have $\vert P_\ell-P_\ell'\vert< \sigma n^2=\lambda/4$ for any outcomes of the good variables $x_i\in \{-1,1\}$ (this is because at most $n^2$ terms get deleted in $P_\ell'$, each with absolute value less than $\sigma$).

We say that a good variable is \emph{short-popular} if it appears in the polynomial $P_\ell'$ for at least $m^{1/3}$ different short indices $\ell\in \{1,\dots,m\}$. Note that there are at most $m\cdot 6m^{1/6}/m^{1/3}=6m^{5/6}$ short-popular variables. For the remainder of the proof of this claim, let us now also condition on any fixed outcomes for all short-popular variables $x_i$, which we no longer view as being random.

Since there are at most $2m^{5/6}$ bad variables, and at most $6m^{5/6}$ short-popular variables, there are at least $m/3-8m^{5/6}\geq m/4$ short indices $\ell\in \{1,\dots,m\}$ for which both of the variables $x_{i_\ell}$ and $x_{j_\ell}$ are good and not short-popular (for the inequality here, we are assuming that $m$ is sufficiently large). For any such index $\ell$, the coefficient of $x_{i_\ell}x_{j_\ell}$ in the polynomial $P_\ell'$ has absolute value at least $\lambda/2$ (since this is also the case for $P_\ell$, and $\lambda/2>\sigma$). Hence we can apply \cref{fact:non-degenerate} to find that $\Pr(|P_{\ell}'|\ge\lambda/2)\ge 1/4$.

So, if $Y$ is the number of short indices $\ell$ with $|P_{\ell}'|\ge\lambda/2$, then $\E Y\ge(m/4)\cdot (1/4)=m/16$. Recall that we already conditioned on outcomes of all the short-popular variables. So each of the remaining random variables occur in $P_\ell'$ for at most $m^{1/3}$ short indices $\ell$, and hence varying each individual variable affects the value of $Y$ by at most $m^{1/3}$. Therefore, by \cref{lem:AH} we have $Y\ge m/32$ with probability at least
\[1-2\exp\left(-\frac{(m/32)^{2}}{2\cdot 2m\cdot m^{2/3}}\right)=1-2\exp\left(-\frac{m^{1/3}}{2^{12}}\right)\geq 1-m^{-1/24},\]
assuming that $m$ is sufficiently large. Hence with probability at least $1-m^{-1/24}$ there are at least $m/32$ short indices $\ell$ with $|P_{\ell}'|\ge\lambda/2$, which implies that $|P_{\ell}|\ge\lambda/2-\lambda/4\geq \lambda/(4n^4)$. This in particular proves \cref{eq-many-l-satisfy}.
\end{proof}

We may from now on assume that there are at least $m/3$ indices $\ell\in \{1,\dots,m\}$ which are not easy and not short. Let us call such indices \emph{interesting}.

For every interesting index $\ell\in \{1,\dots,m\}$, recall that $P_\ell$ is a multilinear polynomial in the variables $x_i$ with $i\in I_{\text{good}}$. Furthermore, recall that for any distinct $i,j\in I_{\text{good}}\setminus S_\ell$ the coefficient of $x_ix_j$ in $P_\ell$ has absolute value less than $\tau$. Let $P_\ell^*$ be the polynomial obtained from $P_\ell$ by deleting all terms of the form $x_ix_j$ for $i,j\in I_{\text{good}}\setminus S_\ell$. Note that we always have $\vert P_\ell-P_\ell^*\vert\leq \tau n(n-1)/2\leq \tau (n^2-1)=\sigma-\tau$ (for all outcomes of the $x_i$).

For every interesting $\ell\in \{1,\dots,m\}$, there are at least $6m^{1/6}$ good variables which appear in a term of $P_\ell$ whose coefficient has absolute value at least $\sigma$. Since only terms with coefficient less than $\tau<\sigma$ get deleted in $P_\ell^*$, this means that there are also at least $6m^{1/6}$ good variables which appear in a term of $P_\ell^*$ whose coefficient has absolute value at least $\sigma$.

Now, for every interesting $\ell\in \{1,\dots,m\}$, let us interpret $P_\ell^*\in \mathbb{R}[x_i, i\in I_{\text{good}}]$ as a polynomial $Q_\ell\in \mathbb{R}[x_i, i\in S_\ell][x_i, i\in I_{\text{good}}\setminus S_\ell]$, i.e.\ as a polynomial in the variables $x_i$ for $i\in I_{\text{good}}\setminus S_\ell$ whose coefficients are polynomials in the variables $x_i$ for $i\in S_\ell$. Then $Q_\ell$ is a linear polynomial (in the variables $x_i$ for $i\in I_{\text{good}}\setminus S_\ell$). Its constant coefficient is a quadratic polynomial, and its other coefficients are linear polynomials (in the variables $x_i$ for $i\in S_\ell$).

Let $T_\ell$ be the number of degree-$1$ coefficients of the linear polynomial $Q_\ell$ (in the variables $x_i$ for $i\in I_{\text{good}}\setminus S_\ell$) which have absolute value at least $\sigma$. This is a random variable depending on the outcomes of the $x_i$ with $i\in S_\ell$.

\begin{claim}\label{claim-T-ell-large}
For each interesting index $\ell\in \{1,\dots,m\}$, we have $\Pr\left(T_\ell\geq m^{1/6}\right)\geq 1/3$.
\end{claim}
\begin{proof}
Fix an interesting $\ell\in \{1,\dots,m\}$. Recall that there are at least $6m^{1/6}$ good variables $x_i$ which appear in a term of $P_\ell^*$ whose coefficient has absolute value at least $\sigma$. Since $\vert S_\ell\vert\leq 2m^{1/6}$, at least $4m^{1/6}$ of these variables satisfy $i\in I_{\text{good}}\setminus S_\ell$. For each such $i$, the coefficient of $x_i$ in $Q_\ell$ is a linear polynomial in the variables $x_j$ for $j\in S_{\ell}$, with at least one coefficient of absolute value at least $\sigma$. Thus, by \cref{fact:non-degenerate-linear}, with probability at least $1/2$ the coefficient of $x_i$ in $Q_\ell$ has absolute value at least $\sigma$. By Markov's inequality (specifically \cref{lem:Markov} applied with $p=1/2$ and $q=1/4$), with probability at least $1/3$ there are at least $(1/4)\cdot 4m^{1/6}=m^{1/6}$ different $i\in I_{\text{good}}\setminus S_\ell$ such that the coefficient of $x_i$ in $Q_\ell$ has absolute value at least $\sigma$. In other words, with probability at least $1/3$, we have $T_\ell\geq m^{1/6}$. 
\end{proof}

\begin{claim}\label{claim-T-ell-large-and-poly-small}
For each interesting index $\ell\in \{1,\dots,m\}$, we have $\Pr\left(T_\ell\geq m^{1/6}\text{ and }\vert P_\ell^*\vert<\sigma\right)\leq 3m^{-1/12}$.
\end{claim}
\begin{proof}
Fix an interesting $\ell\in \{1,\dots,m\}$. Recall that $T_\ell$ depends only on the variables $x_i$ with $i\in S_\ell$. So, for the proof of this claim, let us condition on some outcome for the variables $x_i$ with $i\in S_\ell$ such that we have $T_\ell\geq m^{1/6}$. Under this conditioning the polynomial $P_\ell^*$ becomes precisely the polynomial $Q_\ell$, which is a linear polynomial in the remaining variables $x_i$ for $i\in I_{\text{good}}\setminus S_\ell$, having $T_\ell\geq m^{1/6}$ degree-$1$ coefficients with absolute value at least $\sigma$. Now, by the Erd\H os--Littlewood--Offord inequality (\cref{lem:LO}, applied with $t=1$) we indeed have $\Pr\left(\vert P_\ell^*\vert<\sigma\right)\leq 3/T_\ell^{1/2}\leq 3m^{-1/12}$.
\end{proof}

Next, define the random variable $X$ as the number of interesting $\ell\in \{1,\dots,m\}$ such that $T_\ell\geq m^{1/6}$. Furthermore, define the random variable $Y$ as the number of interesting $\ell\in \{1,\dots,m\}$ such that $T_\ell\geq m^{1/6}$ and $\vert P_\ell^*\vert<\sigma$.

\begin{claim}
With probability at least $1-m^{-1/12}$ we have $X\geq m/18$.
\end{claim}
\begin{proof}
First, note that by \cref{claim-T-ell-large} we have $\E X\geq (m/3)\cdot (1/3)=m/9$. Also recall that for each interesting $\ell$, the event $T_\ell\geq m^{1/6}$ only depends on the outcomes of $x_i$ for $i\in S_\ell$. This means that changing the outcome of any of the good random variables $x_i$ can affect $X$ by at most $m^{1/3}$ (recall that for each good $x_i$, we have $i\in S_\ell$ for at most $m^{1/3}$ different $\ell$). Hence by \cref{lem:AH}, we have $X\geq m/18$ with probability at least
\[1-2\exp\left(-\frac{(m/18)^{2}}{2\cdot 2m\cdot m^{2/3}}\right)=1-2\exp\left(-\frac{m^{1/3}}{36^2}\right)\geq 1-m^{-1/12}\]
(for sufficiently large $m$), as desired.
\end{proof}

\begin{claim}
With probability at least $1-108m^{-1/12}$ we have $Y\leq m/36$.
\end{claim}
\begin{proof}
By \cref{claim-T-ell-large-and-poly-small} we have $\E Y\leq m\cdot 3m^{-1/12}=3m^{11/12}$. Hence, by Markov's inequality we have $Y\geq m/36$ with probability at most $108m^{-1/12}$.
\end{proof}

From the previous two claims, we conclude that if $m$ is sufficiently large, then with probability at least $1-109m^{-1/12}\geq 1-m^{-1/24}$ we have $X-Y\geq m/36$. But note that whenever $X-Y\geq m/36$, there are at least $m/36$ interesting $\ell\in \{1,\dots,m\}$ such that $\vert P_\ell^*\vert\geq \sigma$, which implies $\vert P_\ell\vert\geq \sigma-(\sigma-\tau)= \tau=\lambda/(4n^4)$. This proves \cref{eq-many-l-satisfy}, and finishes the proof of \cref{lem:endgame-step}.

\section{Concluding remarks}\label{sec:concluding}

We have proved that the permanent of a random symmetric $\pm 1$ matrix typically has magnitude $n^{n/2+o(n)}$. This encapsulates the upper bound in \cref{thm:var} and the lower bound in \cref{thm:main}.

The upper bound and lower bound both permit some fairly immediate generalisations. For example, in the setting of \cref{thm:main} we actually have $\Pr(\per M_{n}=a)\le \Pr(|\per M_{n}-a|\le n^{n/2-\varepsilon n})\le n^{-c}$ for all $a$ (not just $a=0$). To see this, recall that the proof of \cref{thm:main} concludes by repeatedly applying \cref{lem:endgame-step}, where the final application shows that, conditional on a typical outcome of $\per M_{n-1}$, it is very likely that $|\per M_{n}|\le n^{n/2-\varepsilon n}$. In this final application we can instead apply a slight generalisation of \cref{lem:endgame-step} (proved in the same way), showing that for any $a$ in fact it is very likely that $|\per M_{n}-a|\le n^{n/2-\varepsilon n}$.

We can also permit the entries of our random matrix to take more general distributions. As defined in the introduction, consider any real probability distributions $\mu$ and $\nu$, and let $M_n^{\mu,\nu}$ be the random symmetric matrix whose diagonal entries have distribution $\nu$ and whose off-diagonal entries have distribution $\mu$ (and whose entries on and above the diagonal are mutually independent). If $\mu$ and $\nu$ are fixed (not depending on $n$), $\nu$ has finite second moment and $\mu$ has vanishing first moment and finite fourth moment, then essentially the same proof as for \cref{thm:var} shows that $\Pr\left(|\per M^{\mu,\nu}_{n}|\ge n^{n/2+\varepsilon n}\right)\le n^{-\varepsilon n}$ for $n$ sufficiently large with respect to $\eps$, $\mu$ and $\nu$.

The conclusion of \cref{thm:main} can be even more freely generalised to any fixed distributions $\mu$ and $\nu$ such that $\mu$ is supported on at least two points (not requiring any moment assumptions at all). In the case where $\mu$ is supported on two points, any quadratic polynomial in independent $\mu$-distributed random variables can be rewritten as a multilinear polynomial. One can then use essentially the same proof as for \cref{thm:main} (only changing the various constants in the lemma statements) to prove that there is a constant $c$ (depending on $\mu$ but not $\nu$) such that for all $\eps>0$ and all $n$ sufficiently large with respect to $\eps$, we have $\Pr(|\per M^{\mu,\nu}_n|\le n^{n/2-\varepsilon n})\le n^{-c}$. Note that changing $\nu$ has no effect on our arguments, because we never actually use the randomness of the diagonal entries. One needs some slight generalisations of the anti-concentration lemmas in \cref{sec:tools} for $\mu$-distributed random variables, but these are indeed available; see \cite[Theorem~A.1]{BVW10} and \cite[Theorem~1.8]{MNV16}.

In the case where $\mu$ is supported on more than two points, then it is necessary to make slightly more involved changes to the proof of \cref{thm:main}, but the same result does hold. To give a brief sketch: in this case, the main issue is that in the proof of \cref{lem:endgame-step} we are no longer able to assume that the relevant quadratic polynomials are multilinear, so we must treat the square terms $x_i^2$ in basically the same way we treat the linear terms $x_i$. To be specific, in the proof of \cref{lem:endgame-step}, we must allow the polynomials $Q_\ell$ to contain square terms $x_i^2$ in addition to linear terms. There are several aspects to this. First, we need to generalise \cref{fact:non-degenerate} and \cref{lem:MNV} to quadratic polynomials of independent $\mu$-distributed random variables. In particular, we need a generalisation of \cref{lem:MNV} for quadratic polynomials that are not necessarily multilinear, still giving a bound in terms of the graph matching number $\nu(G^{(r)}(f))$ (where we ignore the square terms in $f$ for the purpose of constructing the graph $G^{(r)}(f)$). Suitable generalisations of \cref{fact:non-degenerate} and \cref{lem:MNV} can be proved with the methods in \cite[Section~4]{MNV16} (in fact, in this section the authors prove \cite[Theorem~1.8]{MNV16}, which is essentially the required generalisation of \cref{lem:MNV} but with slightly different assumptions).

Second, we need a generalisation of \cref{fact:non-degenerate-linear} for $\mu$-distributed random variables (for which we can make fairly trivial changes to the proof of \cref{fact:non-degenerate-linear}). And lastly, we need a generalisation of the Erd\H os--Littlewood--Offord theorem (\cref{lem:LO}) for polynomials of independent $\mu$-distributed random variables (when $\mu$ is supported on at least three values) which applies not just to linear polynomials, but also to quadratic polynomials consisting of both linear and square terms (having no multilinear degree-2 terms). Such polynomials can be interpreted as linear polynomials in independent (but not identically distributed) random variables, and therefore the appropriate generalisation of \cref{lem:LO} follows from the Doeblin--L\'evy--Kolmogorov--Rogozin inequality~\cite{Rog61} (or alternatively the method in \cite[Section~4]{MNV16}), together with a basic single-variable quadratic anti-concentration bound. Namely, we need the fact that for any real random variable $x$ supported on at least 3 values, and any $r>0$, there are $s>0$ and $p>0$ such that we have $\Pr( |f(x)| < s ) < 1-p$ for any one-variable quadratic polynomial $f$ whose linear or quadratic coefficient has absolute value at least $r$.

So, for example, both \cref{thm:var} and \cref{thm:main}, and therefore \cref{thm:main-0}, can be generalised to the case where $\mu$ is a centred Gaussian distribution (as long as $\mu$ and $\nu$ do not depend on $n$, and $\nu$ has finite second moment). The Gaussian Orthogonal Ensemble (GOE) is an important special case. Another case that may be of particular interest is the case where the support of $\mu$ is  $\{0,1\}$, and $\nu$ is the trivial distribution always taking the value zero (\cref{thm:var} does not hold in this case, but \cref{thm:main} does). In this case $M^{\mu,\nu}_n$ can be interpreted as the adjacency matrix of a random graph. However, we note that in this case the statement in \cref{thm:main} can be proved in a much simpler way: we can take advantage of the fact that changing any off-diagonal entry in $M^{\mu,\nu}_n$ from $0$ to $1$ typically causes a large increase in the value of $\per M^{\mu,\nu}_n$, and apply a more general anti-concentration inequality for functions with this property \cite[Theorem~1.2]{FKS}. Actually, in this setting we suspect that it is possible to prove a limit law for $\per M^{\mu,\nu}_n$, using the ideas in \cite{Jan94}.

Regarding further directions for research, it would be very interesting to prove stronger upper bounds for the concentration probabilities of the permanent, in both the i.i.d.\ and the symmetric case. It is currently only known that $\Pr(\per M_{n}=a)$ and $\Pr(\per A_{n}=a)$ are bounded by $n^{-c}$ for some constant $c$. It would be very interesting to prove bounds of the form $n^{-\omega(1)}$, where $\omega(1)$ is any function going to infinity with $n$. Actually, Vu conjectured (see \cite[Conjecture~6.12]{Vu20}) that $\Pr(\per A_{n}=0)$
is of the form $\omega(1)^{-n}$ (in contrast to the situation for the determinant, where we have $\Pr(\det A_{n}=0)=2^{-n+o(n)}$). It seems reasonable to conjecture that in fact all probabilities of the form $\Pr(\per M_{n}=a)$ or $\Pr(\per A_{n}=a)$ are upper-bounded by $n^{-cn}$, for some constant $c$. However, essentially all known tools for studying permanents of random matrices also apply to determinants, so significant new ideas would be required to prove such strong results.

\textbf{Acknowledgements.} We would like to thank Asaf Ferber for insightful discussions. We are also grateful to the referee for their careful reading of the paper, and their useful comments and suggestions.

\bibliographystyle{amsplain_initials_nobysame_nomr}
\bibliography{ref}

\end{document}